\documentclass{amsart}

\usepackage[all]{xy}
\usepackage{epic,eepic}
\usepackage{mathrsfs}
\usepackage{kbordermatrix}
\usepackage{url}
\usepackage{xcolor}
\usepackage[normalem]{ulem}
\usepackage{float}

\usepackage{verbatim} 

\newcommand{\C}{{\mathbb{C}}}

\newcommand{\N}{{\mathbb{N}}}

\newcommand{\PP}{{\mathbb{P}}}
\newcommand{\R}{{\mathbb{R}}}
\newcommand{\Z}{{\mathbb{Z}}}

\newcommand{\A}{{\mathscr{A}}}

\newcommand{\mr}{M_{\mathbb{R}}}

\newcommand{\ph}{\phantom{-}}
\newcommand{\oDelta}{{\Delta^{\text{\raisebox{3.1pt}{\tiny$\hskip-6.15pt\circ$}}}}}

\newcommand\res[1]{{\lower1.5pt\hbox{$|$}}_{\raise.5pt\hbox{${\scriptstyle #1}$}}}

\newtheorem{theorem}{Theorem}[section]
\newtheorem{lemma}[theorem]{Lemma}
\newtheorem{corollary}[theorem]{Corollary}
\newtheorem{conjecture}[theorem]{Conjecture}
\newtheorem{proposition}[theorem]{Proposition}\theoremstyle{definition}  
\newtheorem{proposition/definition}[theorem]{Proposition/Definition}\theoremstyle{definition}  
\newtheorem{remark}[theorem]{Remark}
\newtheorem{example}[theorem]{Example}
\newtheorem{question}[theorem]{Question}
\newtheorem{definition} [theorem] {Definition} 
 
\numberwithin{theorem}{section}
\numberwithin{equation}{section}

\begin{document}

\title[Moment Maps, Linear Precision, and ML Degree]{Moment Maps, Strict Linear Precision, and Maximum Likelihood Degree One}

\author{Patrick Clarke}
\address{Department of Mathematics, Drexel University, Philadelphia, PA 19104}
\email{pclarke@math.drexel.edu}

\author{David A. Cox}
\address{Department of Mathematics and Statistics, Amherst College, Amherst, MA 01002}
\email{dacox@amherst.edu}

\begin{abstract}
We study the moment maps of a smooth projective toric variety.  In particular, we characterize when the moment map coming from the quotient construction is equal to a weighted Fubini-Study moment map.  This leads to an investigation into polytopes with strict linear precision, and in the process we use results from and find remarkable connections between Symplectic Geometry, Geometric Modeling, Algebraic Statistics, and Algebraic Geometry. 
\end{abstract}


\date{\today}

\maketitle


\section{Introduction}
\label{intro}

This paper began with two moment maps associated to a smooth projective toric variety $X_P$, where $P \subseteq M_\R \simeq \R^d$ is a full-dimensional lattice polytope:\ the quotient moment map and the Fubini-Study moment map.  The first comes from the symplectic geometry of the quotient presentation \cite{cox-homogeneous}
\[
X_P = \mathbb{C}^{\Sigma(1)} /\hskip-2pt/ G.
\]
Specifically, we consider the quotient moment map $\mu_\text{quot} : X_P \to M_\mathbb{R}$ associated to the $(S^1)^d$-action on $X_P$ when equipped with the symplectic structure $\varpi_\text{quot}$ that descends from the standard symplectic form on $\mathbb{C}^{\Sigma(1)}$ via symplectic reduction \cite{marsden-weinstein, meyer}.  It is natural to ask if there is an explicit formula for $\mu_\text{quot}$.

The goal of understanding $\mu_\text{quot}$ might be surprising. In algebraic geometry, there is a familiar formula for ``the'' moment map of a toric variety given by the polytope $P$ in terms of its lattice points $\A = P\cap M$ \cite[\S12.2]{cls}:
\begin{equation}
\label{muFS}
x \longmapsto \frac{1}{\sum_{m \in \A} |\chi^m(x)|^2} \sum_{m \in \A} |\chi^m(x)|^2 m.
\end{equation}
This formula is simple and explicit, however it is not generally equal to $\mu_\text{quot}$. In fact, there are many moment maps, each depending on the symplectic form used on $X_P$. The formula \eqref{muFS} is the moment map $\mu_\text{FS}$ associated to the pullback of the Fubini-Study form on projective space along the embedding 
\[
x \longmapsto [\chi^m(x)]_{m \in \A}.
\]
One can begin see differences between the Fubini-Study moment map 
and the quotient moment map even in the case of projective space.

\begin{example}
\label{P1-FS-vs-quot-example}
$\mathbb{P}^1$, obtained by symplectic reduction so that its polytope is $P = [0,2] \subseteq \mathbb{R}$, gives a simple example for which there are differences between the Fubini-Study moment map and the quotient moment map. Here, $P\cap M = [0,2]\cap\Z = \{0,1,2\}$, and \eqref{muFS} becomes
\[
\mu_\text{FS}([1 : x]) = \frac{|x|^2+ 2|x^2|^2}{1+|x|^2+|x^2|^2}. 
\]
In contrast, Example \ref{simplex strict} (based on \cite[Ex.\ 3.7]{SG}) and Theorem~\ref{momentprecision} imply that
\[
\mu_\text{quot}([1 : x]) = \frac{2|x|^2+ 2|x^2|^2}{1+2|x|^2+|x^2|^2}.
\]
\end{example}

Admittedly the differences in Example~\ref{P1-FS-vs-quot-example} are small. However, there are polytopes, such as the trapezoid considered in Section~\ref{rationalsection}, where the differences are significant. 

Our paper includes several notable results that put Example~\ref{P1-FS-vs-quot-example} in context and reveal a deep connection to the concept of \emph{strict linear precision} from geometric modeling.  Our main result is Theorem~\ref{XML}, whose proof draws on concepts from Algebraic Geometry, 
Geometric Modeling, and Algebraic Statistics.

One can see the interactions among these many areas of mathematics in the combination of Propositions~\ref{wtstructure} and \ref{Km=tx-proposition}, which give the equation
\[
K_{\mathbf{w}} \circ \mu_\text{quot} = \tau_\A \circ [w \chi]_\A \circ |{\bullet}|^2.
\]
Here one has 
\begin{itemize}
\item Krasauskas's map $K_{\mathbf{w}}$ from Geometric Modeling,
\item the quotient moment map $\mu_\text{quot}$ from Symplectic Geometry,
\item Garcia-Puente \& Sottile's tautological map $\tau_\A$ from Algebraic Statistics,
\item the weighted character map $[w \chi]_\A$ from Algebraic Geometry, and
\item the norm-squared map $|{\bullet}|^2$ of a toric variety from Real Algebraic Geometry.
\end{itemize}
Our main result (Theorem~\ref{XML}) combines these ingredients with Huh's theorem \cite{huh} on the Horn parametrization of very affine varieties of maximum likelihood degree one. To this we add 
 a proof that any Horn parametrization is given by an essentially unique minimal Horn matrix (Proposition~\ref{RedHornMatrixProp}).

In the rest of the Introduction, we review some aspects of Symplectic Geometry, Geometric Modeling, and Algebraic Statistics that establish the context of our paper in more detail.

\medskip

\noindent $\bullet$ {\scshape Symplectic Geometry:} 
A $d$-dimensional lattice polytope $P \subseteq M_\mathbb{R} \simeq \R^d$ defines toric variety $X_P$ with a $(S^1)^d$-action.  When $X_P$ is smooth, there are many ways to equip it with a symplectic form for which this action is Hamiltonian.  One way, described in Proposition \ref{wtmomentprop}, uses symplectic reduction to find a symplectic form $\varpi_\text{quot}$ on $X_P$ whose moment map is $\mu_\text{quot}$.  

Another way to get symplectic forms on $X_P$ is via equivariant embeddings into projective space.  Since $X_P$ is smooth, the characters $\chi^m$ for $m \in \A =P\cap M$ give one such embedding, and using appropriate weights leads to \emph{weighted Fubini-Study forms} $\varpi_{\text{FS},w}$ whose \emph{weighted moment maps} $\mu_{\text{FS},w}$ are weighted versions of the explicit formula given in \eqref{muFS}.  The quotient moment map $\mu_\text{quot}$, on the other hand, is described indirectly, and there is no a priori reason to believe it is given by a nice formula.  

An organizing question that motivates and guides this paper is: 

\begin{question}
\label{Q1}
When is the quotient moment map $\mu_\text{quot}$ equal to a weighted moment map $\mu_{\text{FS},w}$?
\end{question}

\noindent
The answer to this question connects Symplectic Geometry to Geometric Modeling and Algebraic Statistics through the property of linear precision.
\medskip


\noindent $\bullet$ {\scshape Geometric Modeling:} 
For any lattice polytope $P$, Krasauskas defined  \emph{toric blending functions} $\beta_m : P \to \R$, one for each lattice point  $m \in \A = P\cap M$ \cite{krasauskas}. These are natural generalizations of monomials in M\"obius' barycentric coordinates \cite{mobius}.  Allowing positive weights $w_m$  and normalizing leads to natural generalizations of the Bernstein basis polynomials:
\[
\overline{\beta}_{m,w} = \frac{w_m\hskip1pt \beta_m}{\sum_{m \in  \A}w_m\hskip1pt \beta_m}.
\]
These normalized weighted toric blending functions are rational functions that share many of the properties that make 
Bernstein basis polynomials
so useful.
However, there is one important property that is not guaranteed---\emph{linear precision}.

Linear precision has to do with approximation. 
Just as with Bernstein basis polynomials, a function $g : P \to \mathbb{R}$ can be approximated by 
\[
G(p) = \sum_{m \in \A } g(m) \overline{\beta}_{m, w}(p).
\]
The relationship is between the original $g$
and the resulting approximation $G$ depends on the functions $w_m \beta_m$.  In the convenient situation where $G=g$, the approximation is said to reproduce $g$. \emph{Linear precision} is the property that the approximation reproduces all affine linear functions.  

Fortunately, it is not necessary to check that every affine linear function is reproduced to know whether or not the $\overline{\beta}_{m,w}$'s have linear precision.  A necessary and sufficient condition is that the equation
\[
p = 
\sum_{m \in \A}  \overline{\beta}_{m,w} (p)\hskip1pt m
\]
hold for $p \in P$.  Geometrically, this means the map 
\[
K_w : P \longrightarrow P, \quad K_w(p) = \sum_{m \in \A}  \overline{\beta}_{m,w} (p)\hskip1pt m
\]
is the identity. In terms of the original blending functions, this equation is
\begin{equation}
\label{eq:linear-precision}
p = \frac{1}{\sum_{m \in  \A}w_m\hskip1pt \beta_m(p)} 
\sum_{m \in \A} w_m\hskip1pt \beta_m(p)\hskip1pt m.
\end{equation}
This leads to the following natural question:

\begin{question}
\label{Q2}
For which polytopes $P$ do there exist weights $(w_m)_{m \in \mathcal{A}}$ so that \eqref{eq:linear-precision} is satisfied?
\end{question}

\noindent
Polytopes for which weights exist so that \eqref{eq:linear-precision} is satisfied are said to have \emph{strict linear precision}.  
In particular, Krasauskas observed that for arbitrary polytopes and arbitrary positive weights the map $K_w$ is an analytic automorphism of $P$ \cite{krasauskas}.  
He called this \emph{analytic precision}, since this means \eqref{eq:linear-precision} is satisfied by the analytic functions 
\[
\tilde{\beta}_m = \beta_m \circ K_w^{-1}.
\]

Our main result (Theorem~\ref{XML}) reveals that Questions~\ref{Q1} and \ref{Q2} are intimately related.  But to understand the full story, we need tools from one more area of mathematics.
 
\medskip

\noindent $\bullet$ {\scshape Algebraic Statistics:}
The lattice points  $\A$ and weights $w$ define a morphism
\[
[w_m \chi^m ]_{m \in \A} : X_P \longrightarrow \mathbb{P}^{s-1}, \quad s = |\A|.
\]
Log-linear statistical models use this embedding to treat the image of the positive points $(X_P)_{>0}$  as a space of probability distributions on $\A$.  
Given an observed probability distribution $u \in \mathbb{P}^{s -1}_{>0}$, \emph{the maximal likelihood estimate} $L(u)$ is the ``best'' estimate in $(X_P)_{>0}$ of $u$.  Darroch and Ratcliff \cite{DR} show that $u$ and $L(u)$ have the same image under the tautological map 
\[
\tau_\A : \mathbb{P}^{s -1}_{\ge0} \longrightarrow  \mr. 
\]
that comes from the linear projection $\tau_\A : \PP^{s-1} \dashrightarrow \PP^d$ that sends $[x_m]_{m\in\A}$ to $[\sum_{m\in\A} x_m:\sum_{m\in\A} x_m\hskip1pt m]$ (see Definition \ref{tau-definition}).  The degree of $\tau_\A$ restricted to the image of $X_P$ in $\PP^{s-1}$ is the \emph{maximum likelihood degree}.  An especially nice case is when the ML degree is one, since $L(u)$ is a rational expression in $u$ when this happens.  
Furthermore, a theorem of Huh \cite{huh} guarantees that when the ML degree is one, $X_P$ is a reduced $A$-discriminantal variety and has a \emph{Horn parametrization}.  

Garcia-Puente and Sottile \cite{SG} connect this to Geometric Modeling by noting that having ML degree one is equivalent to a property they call \emph{rational linear precision}.  This includes strict linear precision as a special case.

Combining all of the these ingredients, we get our main theorem:

\medskip

\noindent
{\bf Theorem~\ref{XML}}
{\bf (Main Theorem).}
{\it For a lattice polytope $P$ and positive weights $w = (w_m)_{m\in\A}$, the following are equivalent:
\begin{enumerate}
\item $P$ has strict linear precision for $w$. 
\item The sum of the facet normals of $P$ equals zero,  and the polynomial
\[
\beta_w = \sum_{m \in \A} w_m \beta_m
\]
in the denominator of \eqref{eq:linear-precision} is a nonzero constant.
\end{enumerate}
When {\rm(1)} and {\rm(2)} hold, the ML degree is one and the ML estimate has an explicit formula determined by the lattice points of $P$.  Furthermore, when $X_P$ is smooth,  {\rm(1)} and {\rm(2)} are equivalent to 
\begin{enumerate}
\item[(3)] $\mu_\text{\rm quot} = \mu_{\text{\rm FS}, w}$
\end{enumerate}}

\vspace{.1in}

All of the terms used above will be defined carefully in the course of the paper, which is organized as follows.  In Section~\ref{momentsection}, we establish notation and review the relevant toric and symplectic geometry.  We then describe the quotient and weighted moment maps.  The corresponding symplectic forms were studied in 1994 by Guillemin \cite{guillemin}.  We review his work in Section~\ref{guillemin}.
Section~\ref{strictsection} begins with toric blending functions and strict linear precision.  Theorem~\ref{strictequivalence} characterizes when strict linear precision occurs, and Conjecture~\ref{simploidalconj} gives our best guess for which polytopes have this property.  We also prove Theorem~\ref{momentprecision}, which relates Question~\ref{Q1} to strict linear precision.  

Section~\ref{mlesection} reviews the concepts of maximum likelihood estimate and maximum likelihood degree, and in Section~\ref{mle1section}, we focus on ML degree one, where there is a formula for the maximum likelihood estimate.    We also recall a key result of Huh \cite{huh}.  In Section~\ref{hornsection}, we prove Theorem~\ref{XML}, which has Theorem~\ref{strictequivalence} as a corollary.  Finally, Section~\ref{rationalsection} discusses the more general notion of \emph{rational linear precision} due to Garcia-Puente and Sottile \cite{SG}.  Although this takes us away from the symplectic setting of Question~\ref{Q1}, we give two extended examples that illustrate how rational linear precision leads to interesting questions.  An appendix discusses the non-negative and positive parts of a toric variety.

\section{Moment Maps of Smooth Toric Varieties}
\label{momentsection}

In this section, we use standard notation from toric geometry \cite{cls}, including dual lattices $M$ and $N$, characters $\chi^m$ for $m \in M$, cones $\sigma$ of a fan $\Sigma$.  Rays $\rho \in \Sigma(1)$ have ray generators $n_\rho \in \rho \cap N$ and give divisors $D_\rho$ on the toric variety of $\Sigma$.  We follow the convention that $\sum_\rho$ and $\prod_\rho$ mean the sum and product over all $\rho\in\Sigma(1)$.

Consider a full-dimensional lattice polytope $P \subseteq \mr \simeq \R^d$.  Let $\Sigma$ be the normal fan of $P$ and $X_P$ be the associated toric variety.  In this section we assume that $X_P$ is smooth.  The $n_\rho$ are the \emph{facet normals} in the facet presentation 
\begin{equation}
\label{fp}
P = \{p\in \mr \mid \langle p,n_\rho\rangle \ge -a_\rho\ \forall\hskip.5pt \rho \in \Sigma(1)\}
\end{equation}
that plays an important role in what follows.  We will also use the exact sequence
\begin{equation}
\label{ses}
0 \longrightarrow M \stackrel{A}{\longrightarrow} \Z^{\Sigma(1)} \stackrel{Q^*}{\longrightarrow} 
\mathrm{Cl}(X_P) \longrightarrow 0,
\end{equation}
where  $A(m) = (\langle m,n_\rho\rangle)_{\rho \in \Sigma(1)}\in \Z^{\Sigma(1)}$ and $Q^*((b_\rho)_{\rho \in \Sigma(1)})  = \big[\sum_\rho b_\rho D_\rho\big] \in \mathrm{Cl}(X_P)$.


\subsection{The Quotient Construction} We follow \cite{cls}.  Consider the affine space $\C^{\Sigma(1)}$ with coordinate ring $\C[z_\rho] = \C[z_\rho \mid \rho \in \Sigma(1)]$.   Define:
\begin{itemize}
\item The Stanley-Reisner ideal $I =\langle \hskip1pt \mbox{\small$\prod$}_{\rho \notin \sigma(1)} z_\rho \mid \sigma \in \Sigma \big\rangle \subseteq \C[z_\rho]$ of $\Sigma$.

\smallskip

\item The non-semistable locus $J = \mathbf{V}(I) \subseteq \C^{\Sigma(1)}$.

\smallskip

\item The group $G = \{(t_\rho)_{\rho \in \Sigma(1)} \in (\C^*)^{\Sigma(1)} \mid  \mbox{\small$\prod$}_\rho t_\rho^{\langle m,n_\rho\rangle} = 1\  \forall\, m\in M\}$.
\end{itemize}
Since $J \subseteq \C^{\Sigma(1)}$ is invariant under the action of $G \subseteq (\C^*)^{\Sigma(1)}$, we can form the quotient $(\C^{\Sigma(1)} \setminus J)/\hskip-2pt/G$, which in this case is a geometric quotient.  By \cite[Thm.\ 5.1.11]{cls}, there is a natural isomorphism of varieties
\begin{equation}
\label{quotientxp}
X_P \simeq (\C^{\Sigma(1)} \setminus J)/\hskip-2pt/G.
\end{equation}
The variety $\C^{\Sigma(1)} \setminus J$ is toric and described by the fan in $\mathbb{R}^{\Sigma(1)}$
\[
\big\{\mathrm{Cone}(e_\rho \ | \  \rho \in \sigma(1)) \ | \ \sigma \in \Sigma)\big\}.
\]
In addition, the morphism $\C^{\Sigma(1)} \setminus J \to X_P$ is toric.

\subsection{The Symplectic Moment Map} Our reference for symplectic geometry is Audin \cite{audin}.  The basic setting is an even-dimensional real manifold  $Z$ with a nondegenerate closed $2$-form $\varpi$.  When a real Lie group $\Gamma$ acts on $(Z,\varpi)$ by symplectomorphisms, its \emph{moment map}
\[
\mu : Z \to \mathrm{Lie}(\Gamma)^*
\]
has the property that for every $\xi \in \mathrm{Lie}(\Gamma)$, 
\[
d(\langle \mu, \xi \rangle)(v) = \varpi(\alpha_z(\xi), v )
\]
where 
\begin{itemize}
\item $\langle \mu, \xi \rangle : Z \to \mathbb{R}$ is the function given by pairing the Lie algebra covector $\mu(z)$ that depends on $z \in Z$ with the fixed Lie algebra vector $\xi$,
\item 
$\alpha_z(\xi) = \frac{{d}}{{d}t}\res{t=0} \exp(t \xi) \cdot z$,
and 
\item $v  \in T_z Z$.
\end{itemize}
The moment map $\mu$ is well-defined  up to translation in $\mathrm{Lie}(\Gamma)^*$.

A simple example is $Z = \C^d$ with symplectic form 
\[
\varpi = \sum_{j=1}^d dx_j \wedge dy_j = \tfrac{i}2\sum_{j=1}^d dz_j \wedge d\overline{z}_j,\quad z_j = x_j + iy_j.
\]
By  \cite[Ex.\ III.1.7]{audin}, the moment map of the diagonal $\Gamma = (S^1)^d$ action on $\C^d$ is 
\[
\mu_0 : \C^d \longrightarrow \R^d, \quad \mu_0(z_1,\dots,z_d) = (\tfrac12|z_1|^2, \dots,\tfrac12|z_d|^2) \in \R^d.
\]
Here, $\R^d$ is naturally identified with $\mathrm{Lie}((S^1)^d)^*$


\subsection{The Quotient Moment Map} For $\C^{\Sigma(1)}$, the   $(S^1)^{\Sigma(1)}$-action has moment map $\mu_0 : \C^{\Sigma(1)} \to \R^{\Sigma(1)}$ defined by
\[
(z_\rho)_{\rho \in \Sigma(1)} \longmapsto ({\textstyle\frac12}|z_\rho|^2)_{\rho \in \Sigma(1)} \in \R^{\Sigma(1)}.
\]
Since a moment map is only defined up to translation, we will instead use
\begin{equation}
\label{mutrans}
\mu_0((z_\rho)_{\rho \in \Sigma(1)}) = ({\textstyle\frac12}|z_\rho|^2 - a_\rho)_{\rho \in \Sigma(1)} \in \R^{\Sigma(1)}.
\end{equation}

The group $G \subseteq (\C^*)^{\Sigma(1)}$ has the real subgroup 
$\Gamma \subseteq (S^1)^{\Sigma(1)}$.  These groups all act on $\mathbf{C}^{\Sigma(1)} \setminus J$.  For $\Gamma$, we can describe the moment map as follows.

\begin{lemma} 
\label{mudef}
The moment map of the $\Gamma$-action on  $\C^{\Sigma(1)} \setminus J \subseteq \C^{\Sigma(1)}$ is the composition
\[
\mu : \C^{\Sigma(1)} \setminus J \subseteq \C^{\Sigma(1)} \stackrel{\mu_0}{\longrightarrow} \R^{\Sigma(1)} \stackrel{Q^*}{\longrightarrow} \mathrm{Cl}(X_P)_\R,
\]
where $Q^*$ is from \eqref{ses}.
\end{lemma}

\begin{proof}
The map $A$ from \eqref{ses} induces a map 
$\mathrm{Hom}_\Z(\Z^{\Sigma(1)},\C^*) \to \mathrm{Hom}_\Z(M,\C^*)$ whose kernel is $G$. This gives an exact sequence of Lie algebras
\[
0 \longrightarrow \mathrm{Lie}(G)  \longrightarrow \C^{\Sigma(1)} 
\longrightarrow
N_\C \longrightarrow 0.
\]
Working over $\R$, we obtain
\[
0 \longrightarrow \mathrm{Lie}(\Gamma)  \longrightarrow \R^{\Sigma(1)} \longrightarrow
N_\R \longrightarrow 0.
\]
Comparing the dual of this to \eqref{ses}, we see that $\mathrm{Lie}((S^1)^{\Sigma(1)})^* \to \mathrm{Lie}(\Gamma)^*$ can be naturally identified with $\R^{\Sigma(1)} \stackrel{Q^*}{\to} \mathrm{Cl}(X_P)_\R$.

Given this identification, the lemma now follows from the first step of the proof of \cite[Thm.\ VII.2.1]{audin}.
\end{proof}

Because of the translation used in \eqref{mutrans}, $0 \in \mathrm{Cl}(X_P)_\R$ is a regular value of $\mu$.  Then the second step of the proof of  \cite[Thm.\ VII.2.1]{audin} implies that $\mu^{-1}(0)$ is contained in $\C^{\Sigma(1)} \setminus J$, is stable under the action $\Gamma$, and satisfies 
\begin{equation}
\label{symred}
\mu^{-1}(0)/\Gamma \simeq (\C^{\Sigma(1)} \setminus J)/\hskip-2pt/G \simeq X_P.
\end{equation}
The quotient $\mu^{-1}(0)/\Gamma$ is an example of \emph{symplectic reduction} \cite[III.2.f]{audin}.  In this situation, 
the form $\varpi = \frac{i}2\sum_\rho dz_\rho \wedge d\overline{z}_\rho$ on $\C^{\Sigma(1)}$, when restricted to $\mu^{-1}(0)$, descends to a symplectic form $\varpi_\text{quot}$ on $\mu^{-1}(0)/\Gamma = X_P$ whose cohomology class is the class of $D_P = \sum_\rho a_\rho D_\rho$.  The torus of $X_P$ is $T = \mathrm{Hom}_\Z(M,\C^*)$, so that the associated real torus $T_\R$ has dual Lie algebra $\mr$.  Thus we have a moment map
\[
\mu_\text{quot} : X_P = \mu^{-1}(0)/\Gamma \longrightarrow \mr,
\]
which can be described as follows.

\begin{proposition}[Quotient Moment Map]
\label{quotmomentprop}
Assume that $(z_\rho)_{\rho \in \Sigma(1)} \in \mu^{-1}(0)$ maps to $x \in X_P$. Then
\[
\mu_\text{\rm quot}(x) = p, 
\]
where $p \in M_\R$ satisfies 
\[
{\textstyle\frac12}|z_\rho|^2 = \underbrace{\langle p,n_\rho\rangle + a_\rho}_{\text{\rm lattice distance}}\quad \forall\hskip.5pt \rho \in \Sigma(1).
\]
\end{proposition}

\begin{proof}
The quotient construction \eqref{quotientxp}, the symplectic reduction \eqref{symred}, and the moment map described in Lemma~\ref{mudef} can be combined to give the diagram shown in Figure~\ref{diagram}.  

\begin{figure}[h]
\SelectTips{cm}{}
\begin{equation*}
\begin{array}{c}
\xymatrix{
&X_P \ar@{-->}[rrd]^{\rule{15pt}{0pt} \ \ \mu_\text{quot} \,= \, \text{moment map}} && 0 \ar[d]\\
\mu^{-1}(0)/\Gamma \ar@{=}[r] \ar@/^1.5pc/@{--}[rrr]^{\mu_\text{quot}} &(\C^{\Sigma(1)} \setminus J)/G  \ar@{=}[u]& & M_\R \ar[d]^A\\
\mu^{-1}(0) \ar@{^{(}->}[r] \ar[u] \ar@{-->}[urrr]^{\widehat\mu} & \C^{\Sigma(1)} \setminus J \ar[u]\ar@{^{(}->}[r]  &  \C^{\Sigma(1)} \ar[r]^{\mu_0} \ar[rd]^{\mu}& \R^{\Sigma(1)}  \ar[d]^{Q^*}\\
&&& \mathrm{Cl}(X_P)_{\R}\ar[d]\\
&&& 0
}
\end{array}
\end{equation*}
\caption{The Quotient Moment Map of $X_P$}
\label{diagram}
\end{figure}
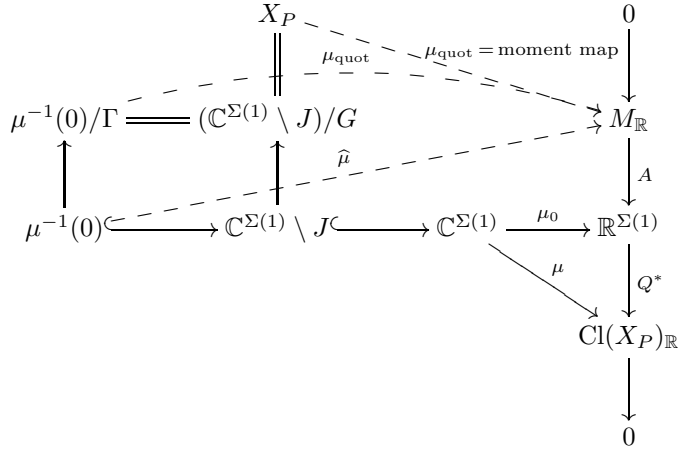

With  $\mu_0((z_\rho)_{\rho \in \Sigma(1)}) = (\frac12 |z_\rho|^2-a_\rho)_\rho$,  Lemma~\ref{mudef} defines
\[
\mu((z_\rho)_{\rho \in \Sigma(1)}) = \big[{\textstyle\sum_{\rho}(\frac12 |z_\rho|^2-a_\rho)D_\rho}\big] \in \mathrm{Cl}(X_P)_\R,
\]
Thus $\mu^{-1}(0)$ consists of points $(z_\rho)_{\rho\in\Sigma(1)}$ satisfying
\[
\big[{\textstyle\sum_{\rho}\frac12|z_\rho|^2 D_\rho}\big] = [D_P].
\]
Given $(z_\rho)_{\rho}\in \mu^{-1}(0)$, the exactness of the vertical sequence in Figure~\ref{diagram} implies that there is a unique $p \in \mr$ such that
\begin{equation}
\label{urhozrho}
\langle p,n_\rho\rangle = {\textstyle\frac12}|z_\rho|^2-a_\rho
\end{equation}
for all $\rho\in\Sigma(1)$.  Then map $\widehat\mu$ in Figure~\ref{diagram} is defined by
\begin{equation}
\label{widehatmu}
\widehat\mu((z_\rho)_{\rho\in\Sigma(1)}) = p.
\end{equation}
Since $\widehat\mu$ is clearly invariant under $\Gamma$, it factors through the quotient, giving the map 
\begin{equation}
\label{moment}
\mu_\text{quot} : X_P = \mu^{-1}(0)/\Gamma \longrightarrow \mr 
\end{equation}
shown in Figure~\ref{diagram}.  This map appears twice in the diagram because we identify $X_P$ with $\mu^{-1}(0)/\Gamma$. 
By \cite[Thm.\ VII.2.1, first step of the proof]{audin} and \cite[Section 2.3]{nick}, this is the moment map of $X_P$ under the action of the real torus $T_\R$.  
\end{proof}

\begin{remark}
We can write the facet presentation \eqref{fp} as
\[
P = \{p\in \mr \mid \langle p,n_\rho\rangle +a_\rho \ge 0\ \forall\hskip.5pt \rho \in \Sigma(1)\}.
\]
For a point $p \in P$, the quantity $ \langle p,n_\rho\rangle +a_\rho$ gives the \emph{lattice distance} from $p$ to the facet of $P$ with facet normal $n_\rho$.  This explains the term ``lattice distance'' that appears in the statement of Proposition~\ref{quotmomentprop}. Lattice distances play a prominent role in some of our main results.

Note that \eqref{urhozrho} implies $\langle p,n_\rho\rangle + a_\rho \ge 0$ for all $\rho$, so that $\mu_\text{quot}(x) = p \in P$. It is a classical theorem of Delzant that the image of the moment map $\mu_\text{quot}$ is precisely the polytope $P$.
\end{remark}

\begin{example} The standard $d$-simplex is
\[
\Delta_d = \{(x_1,\dots,x_n) \in \R^d \mid x_i \ge 0,\ x_1+\cdots+x_d \le 1\} \subseteq \R^d.
\]
The facet normals are $n_0 = -e_1-\cdots-e_d$ and $n_i = e_i$ for $i = 1,\dots,d$, with $a_0 = 1$ and $a_1 = \cdots = a_d = 0$ in the facet presentation.  With $M = \Z^d$, Lemma~\ref{mudef} gives
\[
\mu(z_0,\dots,z_d) = \tfrac12|z_0|^2+\cdots+\tfrac12|z_d|^2 - 1 \in \R,
\]
and when $z = (z_0,\dots,z_d) \in \mu^{-1}(0)$, Proposition~\ref{quotmomentprop} implies the moment map of $\PP^d$ is given by
\[
\mu_\text{quot}(z) = \big(\tfrac12|z_1|^2,\dots,\tfrac12|z_d|^2\big) \in \mr = \R^d.
\]
For later purposes, note that the lattice points 
\[
\A = \Delta_d\cap\Z^d = \{0,e_1,\dots,e_d\} = \{m_0,m_1,\dots,m_d\}
\]
 give the characters $\chi^{m_j}(z) = z_j/z_0$ on the torus of $\PP^d$.  Since $z \in \mu^{-1}(0)$ means $|z_0|^2+ \cdots+ |z_d|^2 = 2$, some algebra reveals that the moment map can be written
\begin{equation}
\label{pdmm}
\begin{aligned}
\mu_\text{quot}(z) &= \frac1{\sum_{j=0}^d |\chi^{m_j}(z)|^2} \sum_{j=0}^d |\chi^{m_j}(z)|^2 \hskip1pt m_j\\ &= \frac1{\sum_{m \in \A} |\chi^{m}(z)|^2} \sum_{m\in \A} |\chi^{m}(z)|^2 \hskip1pt m.
\end{aligned}
\end{equation}
\end{example}

\subsection{Weighted Moment Maps} 
\label{subsection-weighted-moment-maps} 
Besides the quotient moment map described in Proposition~\ref{quotmomentprop}, $X_P$ has moment maps that arise from projective embeddings.  To begin, set $\A = P\cap M$ and $s = |\A|$.  Let $\PP^{s-1}$ be the projective space with homogeneous coordinates indexed by $\A$.  

The characters $\chi^m$ for $m \in \A$ form a basis of the global sections of $\mathcal{O}_{X_P}(D_P)$ and hence give a projective embedding 
\begin{equation}
\label{wt1emb}
[\chi^m]_{m \in \mathcal{A}} : X_P \longrightarrow \PP^{s-1}, \quad x \in X_P \longmapsto [\chi^m(x)]_{m \in \A} \in \PP^{s-1}.
\end{equation}
We know the moment map for $\PP^{s-1}$ from \eqref{pdmm}, so pulling back to $X_P$ suggests that the moment map for $X_P$ should be given by
\begin{equation}
\label{wt1mm}
x \in X_P \longmapsto \frac1{\sum_{m \in \A} |\chi^{m}(x)|^2} \sum_{m\in \A} |\chi^{m}(x)|^2 \hskip1pt m \in \mr.
\end{equation}
This is a convex combination of the lattice points of $P$ and hence lies in $P$. 

In more detail, the Fubini-Study metric on $\PP^{s-1}$ gives a symplectic $2$-form whose cohomology class is $\mathcal{O}_{\PP^{s-1}}(1)$ and has moment map given by \eqref{pdmm}.  Pulling this back via \eqref{wt1emb} gives a symplectic $2$-form whose cohomology class is the first Chern class of $\mathcal{O}_{X_P}(D_P)$ and, as proved in \cite[6.6]{cannas}, the moment map is indeed given by \eqref{wt1mm}.

An important variation is when we assign a positive weight $w_m$ to every lattice point $m \in \A = P\cap M$.  In this case, we get the weighted embedding
\begin{equation}
\label{wtembedding}
[\sqrt{w_m}\hskip1pt \chi^m]_{m \in \mathcal{A}} : X_P \longrightarrow \PP^{s-1}, \quad x \in X_P \longmapsto [\sqrt{w_m}\hskip1pt\chi^m(x)]_{m \in \A} \in \PP^{s-1}.
\end{equation}
The square root appears because $|\sqrt{w_m}\hskip1pt\chi^{m}(x)|^2 = w_m |\chi^{m}(x)|^2$.  Since this map is equivariant with respect to the torus actions on $X_P$ and $\PP^{s-1}$, the arguments of \cite[6.6]{cannas} adapt easily to prove the following.

\begin{proposition}[Weighted Moment Maps]
\label{wtmomentprop}
For any choice of positive weights $w = (w_m)_{m \in \A}$, $X_P$ has a moment map
\[
\mu_{\text{\rm FS}, w}(x) =   \frac1{\sum_{m \in \A} w_m\hskip1pt |\chi^{m}(x)|^2} \sum_{m\in \A} w_m\hskip1pt |\chi^{m}(x)|^2 \hskip1pt m \in \mr.
\]
\end{proposition}

Note that \eqref{wt1mm} is the special case where the weights are all equal to $1$.  This is the version stated in \cite[\S12.2]{cls}, where it is called ``the'' moment map.  We now see that \eqref{wt1mm} is one of infinitely many moment maps of $X_P$.

Here is a way to understand why there are so many.  Multiplying the coordinates of $\PP^{s-1}$ by $(\sqrt{w_m})_{m\in \A}$ is a automorphism of $\PP^{s-1}$, but the Fubini-Study metric is not invariant under this coordinate change.  Thus varying the weights in the embedding of $X_P$ gives a family of symplectic $2$-forms on $X_P$, all cohomologous but each with its own moment map as given in Proposition~\ref{wtmomentprop}.

\subsection{The Motivating Question} Question~\ref{Q1} can now be stated more precisely as asking when are there weights $w = (w_m)_{w\in \A}$ such that the weighted moment map $\mu_{\text{FS}, w}$ of Proposition~\ref{wtmomentprop} is equal to the quotient moment map $\mu_\text{quot}$ of Proposition~\ref{quotmomentprop}.  In Section~\ref{strictsection}, we explore some answers to this question.

\subsection{The Structure of Weighted Moment Maps} 
\label{structuresubsection}
The  formula for $\mu_{\text{FS},w}$ given in Proposition~\ref{wtmomentprop} has a lovely structure we now explore.  This will require three separate maps.  The first is the \emph{norm-squared map}.  In Appendix~\ref{nonnegApp}, we explain how  $\C \to \R_{\ge0}$ defined by $z \mapsto |z|^2$  generalizes to  
\begin{equation}
\label{XPnormsq}
|{\bullet}|^2 : X_P \to (X_P)_{\ge0},
\end{equation}
 where $(X_P)_{\ge0}$ is the \emph{non-negative part} of $X_P$.

The second map comes from a variant of weighted embedding \eqref{wtembedding}, where 
\[
x \in X_P \longmapsto [w_m\hskip1pt\chi^m(x)]_{m \in \A} \in \PP^{s-1},
\]
which we write more compactly as $[w\chi]_\A : X_P \to \PP^{s-1}$.  This map induces (by abuse of notation) a map
\begin{equation}
\label{wtembR}
[w\chi]_\A : (X_P)_{\ge0} \to \PP^{s-1}_{\ge0}
\end{equation}
because the weights are all positive.

For the third map, we begin with a definition.

\begin{definition}
\label{tau-definition}  
Garcia-Puente and Sottile \cite{SG} define the \emph{tautological map}
\[
\tau_\A : \mathbb{P}^{s -1} \dashrightarrow  \PP(\C\oplus M_\C) , \quad \tau_\A\big([x_m]_{m \in \A}\big) = \big[\sum_{m \in \A} x_m : \sum_{m \in \A} x_m\hskip1pt  m \big],
\]
where $\PP(\C\oplus M_\C) \simeq \PP^d$ is the projective space of $\C\oplus M_\C \simeq \C^{d+1}$.   
\end{definition}

Since a point $[x_m]_{m \in \A}$ in $\PP^{s-1}_{\ge0}$ must satisfy $\sum_{m \in \A} x_m > 0$,  $\tau_\A$ induces (by abuse of notation) a map
\begin{equation}
\label{tautologicalR}
\tau_\A : \PP^{s-1}_{\ge0} \longrightarrow M_\R, \quad \tau_\A\big([x_m]_{m \in \A}\big) = \frac{1}{\sum_{m \in \A} x_m} \sum_{m \in \A} x_m\hskip1pt m.
\end{equation}

With these three maps in hand, we get the following formula  for the weighted moment map:

\begin{proposition} 
\label{wtstructure}
The weighted moment map from Proposition~\ref{wtmomentprop} can be written
\[
\mu_{\text{\rm FS},w} = \tau_\A \circ [w\chi]_\A \circ |{\bullet}|^2 : X_P \longrightarrow \mr
\]
using the maps  $|{\bullet}|^2$ from \eqref{XPnormsq},  $[w\chi]_\A$ from \eqref{wtembR}, and $\tau_\A$ from \eqref{tautologicalR}. 
\end{proposition}

\begin{proof}
For $x \in X_P$, $(\tau_\A\circ[w\chi]_\A\circ |{\bullet}|^2)(x) = \tau_\A([w_m \chi^m(|x|^2)]_{m\in\A})$ is equal to 
\[
\frac1{\sum_{m \in \A} w_m\hskip1pt \chi^{m}(|x|^2)} \sum_{m\in \A} w_m\chi^{m}(|x|^2) \hskip1pt m.
\]
As noted in Appendix~\ref{nonnegApp}, $|{\bullet}|^2$ commutes with toric morphisms, so the composition equals
\[
\frac1{\sum_{m \in \A} w_m |\chi^{m}(x)|^2} \sum_{m\in \A} w_m \hskip1pt |\chi^{m}(x)|^2 \hskip1pt m.
\]
This is the formula for $\mu_{\text{\rm FS},w}$ from Proposition~\ref{wtmomentprop}.
\end{proof}

The reader may wonder why \eqref{wtembedding} uses $\sqrt{w_m}$ while \eqref{wtembR} uses $w_m$.  The reason is that $|{\bullet}|^2$ commutes with toric morphisms but \emph{not} with multiplication by weights.  More precisely, we have the commutative diagram
\[
\SelectTips{cm}{}
\xymatrix{
X_P \ar[r]^{[\sqrt{w}\chi_\A]} \ar[d]_{|{\bullet}|^2} & \PP^{s-1} \ar[d]^{|{\bullet}|^2}\\ (X_P)_{\ge0} \ar[r]^{[{w}\chi]_\A} & \PP^{s-1}_{\ge0}
}
\]
It follows that the formula in Proposition~\ref{wtstructure} can be restated as
\[
\mu_{\text{\rm FS},w} = \tau_\A  \circ |{\bullet}|^2 \circ [\sqrt{w}\chi]_\A : X_P \longrightarrow \mr.
\]

\section{Guillemin's Comparison Formula}
\label{guillemin}

In the previous section, we constructed the quotient symplectic form $\varpi_\text{quot}$ on the smooth toric variety $X_P$ and noted that the weighted embedding 
\begin{equation}
\label{weightemb}
[\sqrt{w}\chi]_\A : X_P \longrightarrow \PP^{s-1}, \quad x \in X_P \longmapsto (\sqrt{w_m}\hskip1pt\chi^m(x))_{m \in \A} \in \PP^{s-1}
\end{equation}
gave a symplectic form $\varpi_{\text{FS}, w}$ also on $X_P$.  

Here we state a 1994 formula of Guillemin \cite{guillemin} that relates the forms themselves. This Comparison Formula summarizes the relationship between the symplectic forms associated to the quotient and embedded presentations of the toric variety.


\subsection*{Guillemin's Formula}  To state this result, for $\rho \in \Sigma(1)$ and $p \in \mr$, define
\begin{equation}
\label{lattdist}
h_\rho(p) = \langle p,n_\rho\rangle + a_\rho.
\end{equation}
As noted earlier, this gives the lattice distance from $p \in \mr$ to the facet $F_\rho$ of $P$ with facet normal $n_\rho$.  Also set 
$n_P = \sum_{\rho} n_\rho$ and $h(p) = \langle p,n_P\rangle$ for $p \in \mr$. 

In this notation, (5.8) from Guillemin \cite{guillemin} gives the following theorem.

\begin{theorem}[Comparison Formula]
\label{guilleminthm}
Fix positive  weights $w = (w_m)_{m \in \A}$ as in \eqref{weightemb} and restrict the symplectic forms $\varpi_\text{quot}$ and $\varpi_{\text{FS}, w}$ to the torus $T \subseteq X_P$.  Then
\[
\varpi_{\text{FS}, w}\res{T}\hskip-1pt = \varpi_\text{\rm quot}\res{T} \hskip-1pt+i\partial\overline{\partial}\bigg(\!\hskip-.5pt{-}h\circ\mu_\text{\rm quot}\res{T}\hskip-1pt +
\log\Big( \!\hskip-1pt\sum_{m \in \A} {w}_m \hskip.5pt e^{h(m)}  \prod_{\rho} (h_\rho\circ \mu_\text{\rm quot}\res{T})^{h_\rho(m)}\Big)\!\hskip-1pt\bigg)\hskip-1pt.
\]
\end{theorem}

Note that the log in this expression is defined since $h_\rho$ is positive on the interior $P^\circ$ and the weights are positive.

In the next section, we will relate this formula to Krasauskas' toric blending functions \cite{krasauskas} and apply it to our motivating question in Corollary \ref{guillemincor}.

\section{Strict Linear Precision}
\label{strictsection}

In this section, $P$ will be a full-dimensional lattice polytope in $\mr$ with lattice points $\A = P\cap M$.  

Our main result, Theorem~\ref{XML}, will be proved in Section~\ref{hornsection}.  However, this theorem can be brought to bear on questions about strict linear precision and moment maps in highly non-trivial ways. Therefore we state and use Theorem~\ref{strictequivalence} without proof until Theorem~\ref{XML} for which it is an immediate consequence.

\subsection{Toric Blending Functions}

For $\rho \in \Sigma(1)$ and $p \in \mr$, define
\[
h_\rho(p) = \langle p,n_\rho\rangle + a_\rho.
\]
As noted earlier, this gives the lattice distance from $p \in \mr$ to the facet $F_\rho$ of $P$ with facet normal $n_\rho$.   Here are some definitions due to Krasauskas \cite{krasauskas}.

\begin{definition} \ 
\label{krasauskasblending}
\begin{enumerate}
\item[1.] Given a lattice point $m \in \A$ and a point $p \in P$, define 
\[
\beta_m(p) = \prod_{\rho} h_\rho(p)^{h_\rho(m)}.
\]
\item[2.] Given positive weights ${w} = ({w}_m)_{m \in \A}$, we define 
\[
\beta_{{w}}(p) = \sum_{m \in \A} {w}_m \hskip1pt \beta_m(p).
\]
The functions ${w}_m \hskip1pt \beta_m/\beta_{{w}}$ for $m \in \A$ are the \emph{toric blending functions} of $P$.  
\item[3.] Given \emph{control points} $\{Q_m\}_{m\in \A} \subset \R^n$, we get the \emph{toric patch} 
\[
p \in P \longmapsto \frac{1}{\beta_w(p)}\sum_{m \in \A} w_m \beta_m(p)\hskip1pt Q_m \in \R^n.
\]
\end{enumerate}
We use the convention that $0^0=1$.  Thus implies $h_\rho(p)^{h_\rho(m)} = 0^0 = 1$ when $p, m \in F_\rho$ and ensures that $\beta_m$ and $\beta_w$ are continuous in $p$.
\end{definition}

Note that $\sum_{m\in \A} {w}_m \hskip1pt \beta_m/\beta_{{w}} = 1$ by the definition of $\beta_{{w}}$.  It follows that the image of the toric patch lies in the convex hull of the control points $Q_m$. Toric blending functions and toric patches were introduced by Krasauskas in \cite{krasauskas}. See Garcia-Puente and Sottile \cite{SG} for a discussion of how toric blending functions are used in geometric modeling.

\subsection{Strict Linear Precision} 
A natural choice for the control points in Definition~\ref{krasauskasblending} is $Q_m =m$ for all $m \in \A$.  This gives the following map defined by  Krasauskas \cite{krasauskas}. 

\begin{definition}
\label{K-definition}
Define $K_w : P \to P$ by 
\begin{equation}
\label{K-equation}
K_w(p) =  \frac{1}{\beta_w(p)} 
\sum_{m \in \A} w_m\hskip1pt \beta_m(p)\hskip1pt m.
\end{equation}
\end{definition}

Krasauskas proved that $K_w$ is an analytic isomorphism.  
An important special case 
is when $K_w$ is the identity function:

\begin{definition} 
\label{slpdef}
A full-dimensional lattice polytope $P \subseteq \mr$ has \emph{strict linear precision} for positive weights ${w} = ({w}_m)_{m \in \A}$ provided that
\[
p = K_w(p) \text{ for all } p \in P.
\]
Expanded, this equation is
\[
p = \frac1{\beta_{w}(p)} \sum_{m \in \A} {w}_m \hskip1pt\beta_m(p)\hskip1pt m = \frac1{\sum_{m \in \A} {w}_m \hskip1pt \beta_m(p)} \sum_{m \in \A} {w}_m \hskip1pt \beta_m(p)\hskip1pt m
\]
for all $p \in P$.
\end{definition}

In \cite{SG}, Garcia-Puente and Sottile define a more general notion of linear precision, called \emph{rational linear precision}, that includes strict linear precision as a special case. We will discuss rational linear precision in Section~\ref{rationalsection}.

One of our major results is the following theorem, whose proof will be given in Section~\ref{hornsection}.

\begin{theorem}[Corollary of Theorem \ref{XML}]
\label{strictequivalence}
Let $P \subseteq \mr$ be a full-dimensional lattice polytope with lattice points $\A = P\cap M$ and positive weights $w = (w_m)_{m\in \A}$.  Then the following are equivalent:
\begin{enumerate}
\item $P$ has strict linear precision for $w$. 
\item $\sum_\rho n_\rho =0$ and $\beta_w(p)= \sum_{m\in \A} w_m\hskip1pt \beta_m(p)$ is a nonzero constant.
\end{enumerate}
\end{theorem}

Using this theorem, it is easy to give examples of polytopes with strict linear precision.

\begin{example}[\protect{\cite[Ex.\ 3.7]{SG}}] 
\label{simplex strict}
Let us use Theorem~\ref{strictequivalence} to show that the simplex $k\Delta_d$ has strict linear precision.  The facet normals for $k\Delta_d$ are 
\[
n_0 = -e_1-\cdots-e_d,\ n_1 = e_1, \dots, n_d = e_d,
\]
which obviously sum to zero.  It remains to find weights $w$ so that $\beta_w(p)$ is constant.  Given $m = (a_1,\dots, a_d) \in \A = k\Delta_d\cap\Z^d$, set $|m| = a_1+\cdots+a_d$ and define $\binom{k}{m}$ to be the multinomial coefficient
\[
\binom{k}{m} = \binom{k}{k-|m|,\ a_1, \!\ \dots\ ,  a_n}.
\]
We now compute $\beta_m$.  Given a point $p = (x_1,\dots,x_d) \in k\Delta_d$, one finds that
\begin{align*}
h_0(p) &= \langle p,n_0\rangle + k = k-x_1-\cdots-x_d\\
h_i(p) &= \langle p,n_i\rangle + 0 = x_i,\ i = 1,\dots,d.
\end{align*}
If we set $x_0 = h_0(p) = k - x_1-\cdots-x_d$, then we obtain
\[
\beta_m(p) = h_0(p)^{h_0(m)} h_1(p)^{h_1(m)} \cdots h_n(p)^{h_n(m)} = x_0^{k-|m|} x_1^{a_1} \cdots x_n^{a_n} =  x_0^{k-|m|} x^m.
\]
For weights ${w}_m = \binom{k}{m}$, it follows that
\[
\beta_w(p) = \sum_{m \in \A} {w}_m\hskip1pt \beta_m(p) = 
\sum_{m \in k\Delta_d\cap \Z^d} \binom{k}{m} x_0^{k-|m|} x^m = (x_0+\cdots+x_n)^k = k^k.
\]
By Theorem~\ref{strictequivalence}, $k\Delta_d$ has strict linear precision with weights $w_m = \binom{k}{m}$.  (It is also possible to prove directly that $k\Delta_d$ has strict linear precision for these weights.)
\end{example}

\begin{example} 
\label{productexample} 
Suppose that lattice polytopes $P \subseteq \mr$ and $Q \subseteq M'_\R$ have strict linear precision with weights $w = (w_m)_{m \in P\cap M}$ and $w' = (w'_{m'})_{m' \in Q\cap M'}$.  Let us show that $P\times Q$ has strict linear precision where $(m,m') \in (P\times Q)\cap (M\times M') = (P\cap M)\times (Q\times M')$ has weight $w_m w'_{m'}$.  

The polytopes $P,Q$ have normal fans $\Sigma,\Sigma'$ and facet presentations
\begin{align*}
P &= \{m \in \mr \mid \langle m,n_\rho\rangle \ge -a_\rho,\ \rho \in \Sigma(1)\}\\
Q &= \{m' \in M'_\R \mid \langle m',n'_{\rho'}\rangle \ge -a'_{\rho'},\ \rho' \in \Sigma'(1)\}.
\end{align*}
This makes it easy to write the facet presentation for $P\times Q$, from which we see that the facet normals of $P\times Q$ are 
\[
\big\{ (n_\rho,0) \mid \rho \in \Sigma(1)\big\} \cup \big\{ (0,n'_{\rho'}) \mid \rho' \in \Sigma'(1)\big\} 
\]
Since $\sum_\rho n_\rho = 0 \in M$ and $\sum_{\rho'} n'_{\rho'} = 0 \in M'$ by Theorem~\ref{strictequivalence}, it follows that the facet normals of $P\times Q$ also sum to zero in $M \times M'$.  

Using the facet presentation for $P\times Q$, it is also easy to see that
\begin{equation}
\label{blendproduct}
\beta_{m,m'}^{P\times Q}(p,p') = \beta_m^P(p)\hskip1pt\beta_{m'}^Q(p').
\end{equation}
For the weights $w\times w' = (w_m w'_{m'})_{(m,m') \in (P\times Q)\cap (M\times M')}$, one computes that for $(p,p') \in P\times Q$, 
\begin{align*}
\beta_w^P(p)\hskip1pt\beta_{w'}^Q(p') &= \Big(\sum_{m\in P\cap M} {w}_m \hskip1pt\beta^P_m(p)\Big) \Big(\sum_{m'\in Q\cap M'} {w'}_{m'} \hskip1pt\beta^Q_{m'}(p')\Big)\\
&= \!\!\!\sum_{(m,m') \in (P\times Q)\cap (M\times M')}\!\!\!\! w_m w'_{m'}\hskip1pt\beta_m^P(p)\hskip1pt\beta_{m'}^Q(p') 
=\beta_{w\times w'}^{P\times Q}(p,p'), 
\end{align*}
where the last equality uses \eqref{blendproduct}.  By Theorem~\ref{strictequivalence}, $\beta_w^P$ and $\beta_{w'}^Q$ are nonzero constants.  Hence the same is true for $\beta_{w\times w'}^{P\times Q}$, which by the same theorem implies that $P\times Q$ has strict linear precision for weights $w\times w'$. 
\end{example}

Combining these two examples gives the following result.

\begin{proposition}
\label{exampleprop}
The polytopes $k_1\Delta_{d_1}\!\times \cdots \times k_r\Delta_{d_r}$ have strict linear precision.
\end{proposition}

As noted in  \cite{SG}, the polytopes of Proposition~\ref{exampleprop} are sometimes called \emph{B\'ezier simploids}.  These are the only examples we know of polytopes with strict linear precision.  Theorem~\ref{strictequivalence} shows that strict linear precision is likely to be rare.  This is because while it is easy to give examples with $\sum_\rho n_\rho = 0$, the requirement that $\beta_w$ be constant is much stronger.  Based on this, we make the following conjecture.

\begin{conjecture}
\label{simploidalconj}
The only lattice polytopes with strict linear precision are the B\'ezier simploids described in Proposition~\ref{exampleprop}.
\end{conjecture}

This conjecture is trivial in dimension 1.  In dimension 2, it is a theorem based on the work of Graf von Bothmer, Ranestad and Sottile \cite{GRS}.

\begin{theorem}
\label{dim2slp}
The only lattice polygons $P$ with strict linear precision are the triangles $k\Delta_2$ and the rectangles $k\Delta_1 \times \ell \Delta_1$ for integers $k,\ell \ge 1$.
\end{theorem}

\begin{proof}
Suppose that $P$ has strict linear precision.  Then, as we will see in Section~\ref{rationalsection}, it has rational linear precision, so that by \cite[Corollary 0.3]{GRS}, $P$ is $k\Delta_2$, $k\Delta_1 \times \ell \Delta_1$, or one of the trapezoids 
\[
\mathrm{Conv}(0,(a+db)e_1,be_2,ae_1+be_2)
\]
for integers $a,b,d \ge 1$.  The facet normals of such a trapezoid sum to $-de_2 \ne 0$, so that by Theorem~\ref{strictequivalence}, these trapezoids cannot have strict linear precision.
\end{proof}

\subsection{Relation to Moment Maps}
The connection between moment maps and the blending functions of Definition~\ref{krasauskasblending} is easy to state.   

\begin{proposition}
\label{Km=tx-proposition} 
With $n_P = \sum_\rho n_\rho$ and $\mathbf{w}_m = 2^{\langle m,n_P\rangle} w_m$, 
\begin{equation}
\label{Km=tx-equation}
K_{\mathbf{w}} \circ \mu_\text{\rm quot} = \mu_{\text{FS,w}}.
\end{equation}
\end{proposition}

\begin{proof}
This is essentially a computation using Proposition~\ref{quotmomentprop}.  If $z = (z_\rho)_{\rho\in\Sigma(1)} \in \mu^{-1}(0)$ maps to $x \in X_P$, then we know that $p = \mu_\text{quot}(x)$ satisfies 
\[
\tfrac12|z_\rho|^2 = \langle p,n_\rho\rangle + a_\rho = h_\rho(p).
\]
Suppose in addition that $z \in (\C^*)^{\Sigma(1)}$.  Since the torus of $X_P$ is $\text{Hom}_\Z(M,\C^*)$ and the map  
$(\C^*)^{\Sigma(1)} \to \text{Hom}_\Z(M,\C^*)$ is induced by the map $A$ in \eqref{ses}, it follows that
\[
\chi^m(x) = \prod_\rho z_\rho^{\langle m, n_\rho \rangle}
\]
for all $m \in M$.  

If we set $c(z) = \prod_{\rho}  (\frac12|z_\rho|^2)^{a_\rho}$, then we compute:
\begin{align*} 
|\chi^m(x)|^2 c(z) &=  \prod_\rho |z_\rho^{\langle m, n_\rho \rangle}|^2 \prod_{\rho} (\tfrac12 |z_\rho|^2)^{a_\rho}\\
&= \prod_\rho 2^{\langle m,n_\rho\rangle} \prod_{\rho}(\tfrac12 |z_\rho|^2)^{\langle m, n_\rho \rangle} \prod_{\rho} (\tfrac12 |z_\rho|^2)^{a_\rho}\\
&=  2^{\langle m,n_P\rangle} \prod_{\rho} (\tfrac{1}{2} |z_\rho|^2)^{h_\rho(m)}\\
&=  2^{\langle m,n_P\rangle} \prod_{\rho} h_\rho(p)^{h_\rho(m)}\\
&=  2^{\langle m,n_P\rangle} \beta_m(p), 
\end{align*}
where the last line uses the non-normalized blending function Definition~\ref{krasauskasblending}.   When we combine this with 
$\mathbf{w}_m = 2^{\langle m,n_P\rangle} w_m$ and 
\[
K_{\mathbf{w}} (p)  =  \frac{1}{\sum_{m\in \A} \mathbf{w}_m\hskip1pt\beta_m(p)} \sum_{m\in \A}  \mathbf{w}_m\hskip1pt \beta_m(p)\hskip1pt m,
\]
we obtain
\begin{align*}
K_{\mathbf{w}} (p) &= \frac{1}{\sum_{m\in \A} \mathbf{w}_m\hskip1pt2^{-\langle m,n_P\rangle} |\chi^{m}(x)|^2 \hskip1pt c(z)} 
\sum_{m\in \A}  \mathbf{w}_m\hskip1pt 2^{-\langle m,n_P\rangle} |\chi^{m}(x)|^2 \hskip1pt c(z) \hskip1pt m\\
 &= \frac{1}{\sum_{m\in \A} w_m\hskip1pt |\chi^{m}(x)|^2} 
\sum_{m\in \A}  w_m\hskip1pt |\chi^{m}(x)|^2 \hskip1pt m \\
& =  \mu_{\text{FS},w}(x).
\end{align*}
Since $p = \mu_\text{quot}(x)$, this gives $(K_{\mathbf{w}} \circ \mu_\text{quot})(x) = \mu_{\text{FS},w}(x)$ when $x$ is in the torus of $X_P$.  By continuity, the equality must hold on all of $X_P$.
\end{proof}

Krasauskas \cite{krasauskas} proved that the map $K_{\mathbf{w}}: P \to P$ is analytically invertible.  Hence Proposition~\ref{Km=tx-proposition} gives the analytic formula for the quotient moment map
\[
\mu_\text{quot} = K_{\mathbf{w}}^{-1} \circ \mu_{\text{FS},w}
\]
for any set of positive weights $w$.

Proposition~\ref{Km=tx-proposition} also leads to our first major result as follows.

\begin{theorem}
\label{momentprecision}
Assume that $P$ is smooth with lattice points $\A = P\cap M$.  Then for positive weights $w = (w_m)_{m\in \A}$, the following are equivalent:
\begin{enumerate}
\item The quotient moment map $\mu_\text{\rm quot}$ of $X_P$ is the weighted moment map $\mu_{\text{\rm FS},w}$, i.e., 
\[
\mu_\text{\rm quot}(x) = \frac{1}{\sum_{m\in \A} w_m\hskip1pt |\chi^m(x)|^2} \sum_{m\in \A}  w_m\hskip1pt |\chi^m(x)|^2\hskip1pt m
\]
for all $x \in X_P$.
\item $P$ has strict linear precision for the weights $\mathbf{w}_m = 2^{\langle m,n_P\rangle} w_m$,  $n_P = \sum_\rho n_\rho$, i.e., 
\[
p = \frac{1}{\sum_{m \in \A} \mathbf{w}_m \hskip1pt\beta_m(p)} 
\sum_{m\in \A} \mathbf{w}_m \hskip1pt \beta_m(p)\hskip1pt m
\]
for all $p \in P$.
\end{enumerate}
\end{theorem}

\begin{proof}
Proposition~\ref{Km=tx-proposition} implies $K_{\mathbf{w}} \circ \mu_\text{quot} = \mu_{\text{FS},w}$.  Since $\mu_\text{quot},\, \mu_{\text{FS},w} : X_P \to P$ are surjective, we have the equivalences 
\[
(1) \iff \mu_\text{quot} = \mu_{\text{FS},w} \iff 
K_{\mathbf{w}}\text{ is the identity} \iff (2). \qedhere
\]
\end{proof}

\subsection{Guillemin's  Formula Applied to the Motivating Question} 
\label{AppMotQ}
Using the Krasauskas function $\beta_w$ from Defintion~\ref{krasauskasblending}, Theorem \ref{guilleminthm} has the following nice corollary.

\begin{corollary}
\label{guillemincor}
If $\sum_\rho n_\rho =0$ and $\beta_w$ is constant, then $\mu_\text{quot} = \mu_{\text{FS}, w}$.
\end{corollary}

\begin{proof}
Our hypothesis gives $n_P = \sum_\rho n_\rho = 0$, so that $h(p) = \langle p,n_P\rangle = 0$ for all $p \in \mr$.  Then the Comparison Formula in 
Theorem~\ref{guilleminthm} reduces to
\begin{align*}
\varpi_{\text{FS}, w}\res{T} &= \varpi_\text{quot}\res{T} +i\partial\overline{\partial} \log\Big( \!\sum_{m \in \A} {w}_m  \prod_{\rho} (h_\rho\circ \mu_\text{quot}\res{T})^{h_\rho(m)}\Big)\\
&= \varpi_\text{quot}\res{T} +i\partial\overline{\partial}\log\big(\beta_w\circ\mu_\text{quot}\res{T}\big).
\end{align*}
Since $\beta_w$ is constant, the same is true for $\log\big(\beta_w\circ\mu_\text{quot}\res{T}\big)$, and $\varpi_\text{quot} = \varpi_{\text{FS},w}$ follows immediately.  Since 
 the moment map is determined up to translation by the symplectic form, $\mu_\text{quot}$ and $\mu_{\text{FS}, w}$ are the same up to translation. Equality follows since we have set things up so that both have the polytope $P$ as their image.
\end{proof}

\begin{remark} \
\begin{enumerate}
\item We will see later, that the converse of Corollary \ref{guillemincor} holds, i.e., the \emph{only} way that $\varpi_\text{quot}$ can equal $\varpi_{\text{FS},w}$ is when $\sum_\rho n_\rho =0$ and $\beta_w$ is constant. This is a consequence of part of Main Theorem (Theorem~\ref{XML}) addressing the equality of moment maps.
\item The polynomials $\beta_m$ and $\beta_w$ were defined by Krasauskas in 2002 in the geometric modeling paper  \emph{Toric surface patches} \cite{krasauskas}.  It is fascinating to see a version of these polynomials appearing eight years earlier in the symplectic geometry paper \emph{Kaehler structures on toric varieties} \cite{guillemin}.
\end{enumerate}
\end{remark}

It is interesting to compare Corollary~\ref{guillemincor} and Theorem~\ref{momentprecision} to our main result,  Theorem~\ref{XML} from the Introduction. Corollary~\ref{guillemincor} implies (2) $\Rightarrow$ (3) of Theorem~\ref{XML}, and Theorem~\ref{momentprecision} \emph{almost} implies  (3) $\Rightarrow$ (1).  The discrepancy is that Theorem~\ref{momentprecision} involves weights $w_m$ and $\mathbf{w}_m$.  But these become equal when $\sum_\rho n_\rho = 0$, which is part of (2) of Theorem~\ref{XML}.  

It seems clear that everything will eventually fit together, but we still have more work to do.  To complete the proof of Theorem~\ref{XML}, we need to study the concept of  \emph{maximum likelihood degree} from algebraic statistics.  


\section{Maximum Likelihood Degree}
\label{mlesection}

In this section, we recast our setup in the language of algebraic statistics and define maximum likelihood estimates and the maximum likelihood degree of a very affine variety.  In Section~\ref{mle1section}, we focus on the case when the ML degree is one, which by a theorem of Huh \cite{huh} implies the existence of a Horn parametrization.   Finally, in Section~\ref{hornsection}, we use these tools to prove our main result Theorem~\ref{XML}.

\subsection{A Change in Notation}\label{notationchange} 
In order to write matrices related to the Horn parametrizations, for the rest of the paper we  enumerate facets and lattice points and fix a basis of $M$, so that $M = \Z^d$.  Thus $P$ is now a $d$-dimensional lattice polytope in $\R^d$ with lattice points $\A = P\cap\Z^d$.

We enumerate the facet normals as $n_1,\dots,n_r$, so that the corresponding facets are $F_1,\dots,F_r$ and the lattice distance from $p \in \R^d$ to $F_i$ is
\[
h_i(p) = \langle p,n_i\rangle + a_i,\ \ i=1,\dots,r.
\]
The index $i$ will be reserved for facets and facet normals.  We also enumerate the lattice points of $P$ as $\A = 
\{m_1,\dots,m_s\}$, so that
\[
h_i(m_j) = \text{lattice distance from the $j$th lattice point to the $i$th facet.}
\]
The index $j$ will be reserved for lattice points. 

This allows us to rewrite Definition~\ref{krasauskasblending} as follows.  For $p \in P$, the function $\beta_{m_j}(p)$ becomes
\[
\beta_j(p) = \prod_{i=1}^r h_i(p)^{h_i(m_j)},\quad j=1,\dots,s.
\]
Also, for positive weights $w = (w_1,\dots,w_s)$, the function $\beta_w$ is
\[
\beta_w(p) = \sum_{j=1}^s w_j\hskip1pt \beta_j(p),
\]
and the toric blending functions are $w_1\hskip1pt\beta_1/\beta_w,\dots,w_s\hskip1pt \beta_s/\beta_w$.  

\subsection{A Change in Perspective} We no longer assume that $X_P$ is smooth, so the map $[w\chi]_\A : X_P \to \PP^{s-1}$ from \eqref{weightemb} need not be an embedding.  Hence we replace $X_P$ with its image $X_{\A,w} \subseteq \PP^{s-1}$.  Also, $M = \Z^d$ implies that the torus of $X_P$ is $(\C^*)^d$, so that restricting \eqref{weightemb} to the torus gives 
\[
[w\chi]_\A : (\C^*)^d \longrightarrow X_{\A,w} \subseteq \PP^{s-1}
\]
defined by $t \mapsto [w_1 t^{m_1} : \ldots : w_s t^{m_s}]$.  In this setup, $X_{\A,w}$ is the Zariski closure of the image.  Following \cite{hosten}, we call $X_{\A,w}$ a \emph{scaled projective toric variety}. 

In order to connect transparently with Theorem 1 of Huh \cite{huh} (Theorem~\ref{huhthm} below), we need to work on a certain very affine Zariski open subset of $X_{\A,w}$.   For coordinates $x_1,\dots,x_s$ on $\PP^{s-1}$, let $W = \mathbf{V}(x_1\cdots x_s(x_1+\cdots+x_s)) \subseteq \PP^{s-1}$ and consider the map
\begin{equation}
\label{projveryaffine}
X_{\A,w}\setminus W \longrightarrow (\C^*)^s,\quad 
[x_1: \ldots: x_s] \longmapsto \frac1{x_1+\cdots+x_s}(x_1,\dots,x_s).
\end{equation}
One can check that the image of this map, denoted $Y_{\A,w}$, is closed in $(\C^*)^s$.  We call $Y_{\A,w}$ a \emph{scaled very affine toric variety}.  

When we think of $X_{\A,w}$ and $Y_{\A,w}$ as parametrized varieties, it will be convenient to replace the torus $(\C^*)^d$ with affine space $\C^d$.  This gives parametrizations
\begin{equation}
\label{XYmonomial}
\begin{aligned}
{[w\chi]}_\A &: \C^d \dashrightarrow X_{\A,w} ,\hskip7.5pt t \longmapsto [w_1\hskip1pt t^{m_1}: \ldots : w_s\hskip1pt t^{m_s}] \in \PP^{s-1}\\
\overline{w \chi}_\A &: \C^d \dashrightarrow Y_{\A,w}, \quad t \longmapsto 
\Big(\frac{w_1\hskip1pt t^{m_1}}{\sum_{j=1}^s w_j\hskip1pt t^{m_j}}, \dots, \frac{w_s\hskip1pt t^{m_s}}{\sum_{j=1}^s w_j\hskip1pt t^{m_j}}\Big) \in (\C^*)^s
\end{aligned}
\end{equation}
such that $X_{\A,w}$ and $Y_{\A,w}$ are the Zariski closures of the images in $\PP^{s-1}$ and $(\C^*)^s$ respectively.

\begin{remark}
\label{wtremark}
The group $\R_{>0} \times (\R_{>0})^d$ acts on weights $w = (w_1,\dots,w_s)$ via
\[
(\lambda,t)\cdot w = (\lambda\hskip1pt t^{m_1} w_1\dots, \lambda\hskip1pt t^{m_s} w_s),
\]
and one can check that $w' = (\lambda,t)\cdot w$ satisfies
$X_{\A,w'}=X_{\A,w}$ and $Y_{\A,w'} = Y_{\A,w}$.
\end{remark}

We next give different parametrizations of $X_{\A,w}$ and $Y_{\A,w}$.  First note that $\beta_1,\dots,\beta_s$ and $\beta_w = \sum_{j=1}^s w_j\hskip1pt \beta_j$ are polynomials and hence are defined  on $\C^d$.  

\begin{proposition}
\label{XYparam}
We have the toric blending parametrizations
\begin{align*}
[w \beta]_\A &: \C^d \dashrightarrow X_{\A,w} ,\hskip7.5pt t \longmapsto [w_1\hskip1pt\beta_1(t):\ldots:w_1\hskip1pt\beta_s(t)]\in \PP^{s-1}\\
\overline{w \beta} _\A&: \C^d \dashrightarrow Y_{\A,w} ,\quad t \longmapsto 
\Big(\frac{w_1\hskip1pt\beta_1(t)}{\beta_w(t)},\dots, \frac{w_s\hskip1pt \beta_s(t)}{\beta_w(t)}\Big) \in (\C^*)^s
\end{align*}
such that $X_{\A,w}$ and $Y_{\A,w}$ are the Zariski closures of the images in $\PP^{s-1}$ and $(\C^*)^s$ respectively.
\end{proposition}

\begin{proof}
We use the notation of Section~\ref{notationchange}.  As in
\cite[Ex.\ 3.5]{SG}, $(z_1,\dots,z_r) \in (\C^*)^r$ gives a point $t_0 \in (\C^*)^d$ such that for $j =1,\dots,s$,
\[
\prod_{i=1}^r z_i^{h_i(m_j)} = c\hskip1pt t_0^{m_j}, \quad c = \prod_{i=1}^r z_i^{a_i}.
\]
If $t \in \C^d$ and $h_1(t)\cdots h_r(t) \ne 0$, then applying the above formula to $(z_1,\dots,z_r) = (h_1(t),\dots,h_r(t) )$ gives $t_0 \in (\C^*)^d$ such that for  $j =1,\dots,s$,
\[
\beta_j(t) = \prod_{i=1}^r h_i(t)^{h_i(m_j)} = c\hskip1pt t_0^{m_j}, 
\]
where $c \in \C^*$ is independent of $j$.  It follows that in projective space,
\[
[w_1\hskip1pt\beta_1(t):\ldots:w_1\hskip1pt\beta_s(t)] = [w_1\hskip1pt t_0^{m_1}:\ldots:w_s\hskip1pt t_0^{m_s}] \in X_{\A,w}.  
\]
Furthermore, \cite[Cor.\ 22]{krasauskas} implies that if we restrict $t \mapsto [w_1\hskip1pt\beta_1(t):\ldots:w_1\hskip1pt\beta_s(t)]$ to $P \subseteq 
\C^d$, then the image is $X_{\A,w}\cap\PP^{s-1}_{\ge0}$, which is Zarski dense in $X_{\A,w}$ since the weights $w$ are positive.  Thus the image of $[w\beta]_\A$ is Zariski dense in $X_{\A,w}$.  Composing with \eqref{projveryaffine} gives $\overline{w\beta}_\A : \C^d \dashrightarrow Y_{\A,w}$.  Since \eqref{projveryaffine} gives a birational map $X_{\A,w} \dashrightarrow Y_{\A,w}$, we see that  
image of $\overline{w\beta}_\A$ is Zariski dense in $Y_{\A,w}$. 
\end{proof}

\subsection{Maximum Likelihood Estimates} Let us review the statistical aspects of our situation.  A general overview can be found in  the survey \cite{HS} by Huh and Sturmfels. Consider the open simplex of dimension $s-1$ defined by
\[
\oDelta{}^{s-1} = \{(u_1,\dots,u_s) \in \R_{>0}^s \mid {\textstyle\sum_{j=1}^s} u_j = 1\}.  
\]
Note that projection onto the first $s-1$ coordinates gives a bijection from $\oDelta{}^{s-1}$ to the interior of the standard simplex $\Delta_{s-1} \subseteq \R^{s-1}$.  Also, if $\PP^{s-1}_{>0} \subseteq \PP^{s-1}$ denotes the subset of points in $\PP^{s-1}$ that have positive real homogeneous coordinates, then \eqref{projveryaffine} induces a bijection
\[
X_{\A,w} \cap \PP^{s-1}_{>0} \simeq Y_{\A,w}\cap \R_{>0}^s \subseteq \oDelta{}^{s-1}.
\]
This is  a \emph{log-linear model}, though these days the term \emph{toric model} is also used.

For the toric model $Y_{\A,w}\cap \R_{>0}^s \subseteq \oDelta{}^{s-1}$, suppose that we have normalized data $u = (u_1,\dots ,u_s) \in\oDelta{}^{s-1}$.  The \emph{maximum likelihood estimate} $\hat{x} = L(u)$ of $u$ (ML estimate for short) is the point $\hat{x}$ of the model that ``best explains'' the data $u$.   More precisely, $\hat{x}$ maximizes the function
\begin{equation}
\label{mletomax}
\ell_u(x) = x^u = x_1^{u_1} \cdots x_s^{u_s},\quad x = (x_1,\dots,x_s) \in Y_{\A,w}\cap \R_{>0}^s.
\end{equation}

There is a classical description of the maximum likelihood estimate that works well for our purposes.  Before stating the result, observe that if $(u_1,\dots ,u_s) \in \oDelta{}^{s-1}$, then $\sum_{j=1}^s u_j m_j$ is a positive convex combination of the lattice points of $P$ and hence lies in the interior $P^\circ$ of $P$.
This leads to the map
\begin{equation}
\label{tauAoDelta}
\tau_\A : \oDelta{}^{s-1} \longrightarrow P^\circ,\quad \tau_\A(u_1,\dots,u_s) = \sum_{j=1}^s u_j m_j,
\end{equation}
which (as we will soon see) is a variant of the map $\tau_\A$ defined in \eqref{tautologicalR}.  

\begin{proposition}
\label{mleprop}
For $u \in \oDelta{}^{s-1}$, $L(u)$ is the unique point $\hat{x} \in Y_{\A,w}\cap \R_{>0}^s$ satisfying $\tau_\A(\hat{x}) = \tau_\A(u)$.

\end{proposition}

\begin{proof} As noted in \cite[Sec.\ 4]{SG}, this is proved in \cite[Lem.\ 4]{DR}.
\end{proof}

When the data $u = (u_1,\dots,u_s)$ is positive but not normalized, we normalize it by dividing by $u_+ = u_1+\cdots+ u_s$.  Then $L(u)$ is the unique point $\hat{x} \in Y_{\A,w}\cap \R_{>0}^s$ satisfying $\tau_\A(\hat{x}) = \tau_\A(u/u_+)$.



\subsection{The Maps \boldmath{$\tau_\A$}} Table ~\ref{tauAtable} shows several maps named 
$\tau_\A$.  We have adjusted Definition 
\ref{tau-definition} and equation \eqref{tautologicalR} to the notation of Section~\ref{notationchange}.  

\begin{table}[H]
\begin{tabular}{l|l|l}
1.\ Def.~\ref{tau-definition} & $\tau_\A : \PP^{s-1} \dashrightarrow \PP^d$ & $[x_{1},\dots,x_{s}] \mapsto [\sum_{j=1}^s x_j : \sum_{j=1}^s x_j\hskip1pt m_j]$ \\[2.5pt] \hline
2.\ \eqref{tautologicalR} \rule{0pt}{10.5pt}  & $\tau_\A : \PP^{s-1}_{\ge0} \to P$ & $[x_{1},\dots,x_{s}] \mapsto \sum_{j=1}^s x_i\hskip1pt m_j/(\sum_{j=1}^s x_j) $\\[2.5pt] \hline
3.\ \eqref{tauAoDelta} \rule{0pt}{10.5pt} &
$\tau_\A : \oDelta{}^{s-1} \rightarrow P^\circ$ & $(u_1,\dots,u_s) \mapsto \sum_{j=1}^s u_j\hskip1pt m_j$.  
\end{tabular}
\caption{The Maps $\tau_\A$}
\label{tauAtable}
\end{table}

\vspace*{-5pt}

The relation between maps 1 and 2 in the table is explained in Section \ref{structuresubsection}.  For maps 2 and 3, note that $[x_1: \ldots: x_s] \mapsto \frac1{x_1+\cdots+x_s}(x_1,\dots,x_s)$ used in \eqref{projveryaffine} maps $\PP^{s-1}_{>0}$ bijectively to $\oDelta{}^{s-1}$.  When restricted to positive points, map 2 is this bijection followed by map 3.

It turns out that we need one more variant of $\tau_\A$, the rational map 
\begin{equation}
\label{piAcomplex}
\tau_\A : (\C^*)^s \dashrightarrow \C^d,\quad \tau_\A(x_1,\dots,x_s) = \sum_{j=1}^s \frac{x_j\,}{x_+}\hskip1pt m_j,
\end{equation}
where $x_+ = x_1+\cdots+ x_s$.

\subsection{Maximum Likelihood Degree} By definition, the \emph{maximum likelihood degree} (ML degree for short) of $X_{\A,w}$ is the number of complex critical points of  
\[
\ell_u([x_1:\ldots:x_s]) = \frac{x_1^{u_1}\cdots x_s^{u_s}}{(x_1+ \dots + x_s)^{u_1+\cdots + u_s}}
\]
regarded as a function on $X_{\A,w} \setminus W$, where $W = \mathbf{V}(x_1\cdots x_s(x_1+ \dots + x_s))$ and $u \in \Z_{>0}^s$ is generic.    We refer the reader to \cite{hosten} and \cite{huh} for the details.  

Under the isomorphism $X_{\A,w} \setminus W \simeq Y_{\A,w}$ from \eqref{projveryaffine}, the above formula reduces to $\ell_u(x) = x^u$ from \eqref{mletomax}.  This tells us that the ML estimate $L(u) \in  Y_{\A,w} \cap \R_{>0}^s$ is one of the critical points counted by the ML degree.

We will need the following geometric interpretation of the ML degree that is implicit in \cite[Sec.\ 4]{SG}.  

\begin{proposition}
\label{mldprop/def}
The ML degree of $Y_{\A,w} \subseteq (\C^*)^s$ is the degree of the restriction
\[
\tau_{\A}\raisebox{-2.5pt}{$|$}{\raisebox{-3pt}{\scriptsize$Y_{\A,w}$}} :  Y_{A,w} \dashrightarrow \C^d
\]
for $\tau_\A$ from \eqref{piAcomplex}.  By ``degree'', we mean the number of points in a generic fiber.
\end{proposition}

\begin{proof}
We will follow the approach used in \cite{hosten}.  We start in the projective setting, where the inclusion $\C^d \subseteq \PP^d$ is written
\[
v = \sum_{\ell=1}^d v_i e_i = (v_1,\dots,v_d) \in \C^d \longmapsto [1:v] = [1:v_1:\ldots:v_d] \in \PP^d.
\]
Furthermore, the projective version of \eqref{piAcomplex} is $\tau_\A : \PP^{s-1} \dashrightarrow \PP^d$ defined by
\[
\tau_\A([x_1:\ldots:x_s]) = \Big[1:\sum_{j=1}^s \frac{x_j\,}{x_+}\hskip1pt m_j\Big] = \Big[x_+:\sum_{j=1}^s x_j\hskip1pt m_j\Big] = \sum_{j=1}^s x_j [1:m_j],
\]
which is map 1 from Table~\ref{tauAtable}.
Now set $f = w_1\hskip1pt t^{m_1} + \cdots + w_s\hskip1pt t^{m_s}$ and observe that the equations (2) from \cite[2.2]{hosten} can be written
\begin{equation}
\label{mleeqn}
x_+ = x_+ \lambda f,\ 
(C x)_1 = x_+ \lambda \hskip1pt t_1 \frac{\partial f}{\partial t_1},\ \dots\ ,\ 
(C x)_d = x_+ \lambda \hskip1pt t_d \frac{\partial f}{\partial t_d},
\end{equation}
where $C$ is the $d\times s$ matrix with columns $m_1,\dots,m_s$.  When $x_+ \ne 0$, 
\eqref{mleeqn} holds if and only if $(x_+,(C x)_1,\dots,(C x)_d)$ and $(f,\frac{\partial f}{\partial t_1},\dots,\frac{\partial f}{\partial t_d})$ give the same point in $\PP^d$.   However,
\[
[x_+:(C x)_1:\ldots:(C x)_d] = [x_+:C x] = [x_+:x_1\hskip1pt m_1 + \cdots + x_s\hskip1pt m_s] = \tau_\A(x),
\]
and following \cite[3.14]{SG},
\[
t_\ell \frac{\partial f}{\partial t_\ell} = t_\ell\frac{\partial}{\partial t_\ell} \Big(\sum_{j=1}^s w_j \hskip1pt t^{m_j} \Big) = \sum_{j=1}^s w_j \hskip.5pt \langle m_j,e_\ell\rangle\hskip1pt  t^{m_j},
\]
so that
\[
\Big(t_1 \frac{\partial f}{\partial t_1},\dots,t_d \frac{\partial f}{\partial t_d} \Big) = \sum_{\ell=1}^d t_\ell\frac{\partial f}{\partial t_\ell} e_\ell =\sum_{\ell=1}^d  \Big( \sum_{j=1}^s w_j \hskip1pt \langle m_j,e_\ell\rangle \hskip1pt t^{m_j} \Big) e_\ell = \sum_{j=1}^s w_j\hskip1pt t^{m_j} \hskip1pt m_j.
\]
It follows that
\[
\Big[f:t_1 \frac{\partial f}{\partial t_1}:\ldots:t_d \frac{\partial f}{\partial t_d} \Big] = \sum_{j=1}^s w_j\hskip1pt t^{m_j}[1:m_j] = \tau_\A([w_1\hskip1pt t^{m_1}:\ldots:w_s\hskip1pt t^{m_s}]).
\]
So \eqref{mleeqn} is equivalent to $\tau_\A(u) = \tau_\A(x)$ for $x = [w_1\hskip1pt t^{m_1}:\ldots:w_s\hskip1pt t^{m_s}] \in X_{\A,w}$.  By \cite[Prop.\ 2.4]{hosten}, the ML degree of $X_{\A,w}$ is the number of points  $x \in X_{\A,w} \setminus W$ satisfying \eqref{mleeqn} for generic $u$.  Then the proposition follows by the equivalence just proved and the definition of $Y_{\A,w}$.  
\end{proof}

\section{Maximum Likelihood Degree One}
\label{mle1section}

For a generic choice of weights $w$, the ML degree is  
\[
\text{degree of }  X_{\A,w} = \,\text{normalized volume of } P
\]
by \cite[Cor.\ 2.5]{hosten} or \cite[Thm.\ 3.2]{HS}.  For special choices of $w$, however, the ML degree can be strictly smaller.  This phenomenon is studied in \cite{hosten}.  In some very special cases, there are weights such that the ML degree drops all the way to one. In this situation, we have the following immediate corollary of Proposition~\ref{mldprop/def}.

\begin{corollary} 
\label{mlonecor}
 $Y_{\A.w}$ has ML degree one $\iff$ $\tau_{\A}\raisebox{-2.5pt}{$|$}{\raisebox{-3pt}{\scriptsize$Y_{\A,w}$}} :  Y_{\A,w} \dashrightarrow \C^d$ has degree one $\iff$ $\tau_{\A}\raisebox{-2.5pt}{$|$}{\raisebox{-3pt}{\scriptsize$Y_{\A,w}$}} :  Y_{\A,w} \dashrightarrow \C^d$ is birational.
\end{corollary}

Here is one case where we have an explicit formula for the ML estimate.

\begin{proposition}
\label{mleoneprop}
Assume that $Y_{\A,w}$ has a rational parametrization
\[
 \Psi :\C^d \dashrightarrow Y_{\A,w},\quad  \Psi(t) = \big(  \psi_1(t),\dots,  \psi_s(t)\big)
\]
such that the rational map $\C^d \dashrightarrow \C^d$ defined by
\begin{equation}
\label{composition}
t \longmapsto \sum_{j=1}^s   \psi_j(t)\hskip1pt m_j = (\tau_\A\circ \Psi)(t)
\end{equation}
is birational.  Then $Y_{\A,w}$ has ML degree one, and if $\varphi: \C^d \dashrightarrow \C^d$ is the inverse of \eqref{composition}, 
then the ML estimate of $u= (u_1,\dots,u_s) \in \R_{>0}^s$ is 
\begin{equation}
\label{mlestimateformula}
\begin{aligned}
&L(u) =  \Psi(\varphi(p)) = \big(  \psi_1(\varphi(p)),\dots,  \psi_s(\varphi(p))\big) \in Y_{\A,w}, \ \text{where}\\
&p = \tau_\A(u) = \sum_{j=1}^s \frac{u_j\,}{u_+}\hskip1pt m_j \in \R^d \text{ and } u_+ = \sum_{j=1}^s u_j.
\end{aligned}
\end{equation}
\end{proposition}

\begin{proof}
The map \eqref{composition} is the composition $\big(\tau_{\A}\raisebox{-2.5pt}{$|$}{\raisebox{-3pt}{\scriptsize$Y_{\A,w}$}} \!\circ  \Psi\big)(t)$.  Since this map is birational and the parametrization $ \Psi$ is dominating, it follows that $\tau_{\A}\raisebox{-2.5pt}{$|$}{\raisebox{-3pt}{\scriptsize$Y_{\A,w}$}}$ must have degree one, and then $Y_{\A,w}$ has ML degree one by  Corollary~\ref{mlonecor}. 

Given $u$, Proposition~\ref{mleprop} tells us that the ML estimate is the unique point of $Y_{\A,w} \cap \R_{>0}^s$ that maps to $\tau_\A(u) = p$ via $\tau_{\A}\raisebox{-2.5pt}{$|$}{\raisebox{-3pt}{\scriptsize$Y_{\A,w}$}}$.  The definition of $\varphi$ implies that this unique point is $ \Psi(\varphi(p))$, so that \eqref{mlestimateformula} is indeed the ML estimate of $u$.\end{proof}

Proposition~\ref{XYparam} gives two natural choices for the parametrization of $Y_{\A,w}$:
\begin{align*}
&t \longmapsto \frac{1}{\sum_{j=1}^s w_j\hskip1pt t^{m_j}} \big(w_1\hskip1pt t^{m_1},\dots,w_s\hskip1pt t^{m_s}\big) \ \ \ \hskip1.5pt\text{(monomial parametrization)}\\
&t \longmapsto \frac1{\beta_w(t)}\big(w_1\hskip1pt \beta_1(t),\dots,w_s\hskip1pt \beta_s(t)\big) \quad\quad\quad \text{(toric blending parametrization).}
\end{align*} 
The first is useful for computing examples, and we will see that the second is especially nice when $P$ has strict linear precision.  

\begin{example}
\label{square1}
Consider the unit square in the plane:
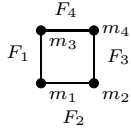
\[
\begin{picture}(52,50)
\put(15,15){\circle*{3.5}}
\put(35,15){\circle*{3.5}}
\put(15,35){\circle*{3.5}}
\put(35,35){\circle*{3.5}}
\put(18,9){\scriptsize $m_1$}
\put(38,9){\scriptsize $m_2$}
\put(18,29){\scriptsize $m_3$}
\put(38,34){\scriptsize $m_4$}
\put(15,15){\line(1,0){20}}
\put(15,15){\line(0,1){20}}
\put(35,35){\line(0,-1){20}}
\put(15,35){\line(1,0){20}}
\put(2,24.5){\scriptsize $F_1$}
\put(23,0){\scriptsize $F_2$}
\put(20,41){\scriptsize $F_4$}
\put(40,23){\scriptsize $F_3$}
\end{picture}
\]
The lattice points are $m_1 = (0,0),\ m_2 = (1,0),\ m_3 = (0,1),\ m_4=(1,1)$.  If we use weights $w = (1,1,1,1)$, then monomial parametrization of $Y_{\A,w}$ is
\[
\overline{w\chi}_A(s,t) = \frac1{1+s+t+st} \big(1,s,t,st\big) \in Y_{\A,w}
\]
in the notation of \eqref{XYmonomial}.  Composing this with $\tau_\A$ gives the map
\begin{align*}
(s,t) \in (\C^*)^2 &\longmapsto \frac1{1+s+t+st}\big( 1\cdot (0,0) + s\cdot (1,0) + t \cdot (0,1) + st\cdot (1,1)\big)\\ &\ =  \frac1{(1+s)(1+t)} (s+st,t+st)=\Big(\frac{s}{1+s},\frac{t}{1+t}\Big).
\end{align*}
This is birational with inverse $\displaystyle\varphi(s,t) = \Big(\frac{s}{1-s},\frac{t}{1-t}\Big)$.  Given $u = (u_1,u_2,u_3,u_4)$ with $1 = u_1 + u_2 + u_3 + u_4$, one calculates that $p = \tau_\A(u) = ({u_2+u_4},{u_3+u_4})$.  Then
\[
\varphi(p) = \Big(\frac{u_2+u_4}{u_1+u_3},\frac{u_3+u_4}{u_1+u_2}\Big),
\]
and the formula \eqref{mlestimateformula} in Proposition~\ref{mleoneprop} implies (after some algebra) that the ML estimate for $u$ is 
\begin{equation}
\label{mlesquare}
\big({(u_1{+} u_3)(u_1{+}u_2)},{(u_2{+}u_4)(u_1{+} u_2)},{(u_3{+} u_4)(u_1{+} u_3)},{(u_2{+} u_4)(u_3{+} u_4)}\big).
\end{equation}

We can also use the blending parametrization of $Y_{\A,w}$, which is 
\begin{align*}
\overline{w\beta}_\A(s,t) &= \frac{\big((1-s)(1-t),s(1-t),(1-s)t,st\big)}{(1-s)(1-t) + s(1-t)+(1-s)t+st} \\
&= \big((1-s)(1-t),s(1-t),(1-s)t,st\big)\in Y_{\A,w}
\end{align*}
in the notation of \eqref{XYparam}.  Composing this with $\tau_\A$ is easily seen to give the identity map, which as we will see in Corollary~\ref{corLisTBP} is a consequence of strict linear precision.  As we saw above, $u = (u_1,u_2,u_3,u_4)$ with $1 = u_1 + u_2 + u_3 + u_4$ gives $p = \tau_\A(u) = ({u_2+u_4},{u_3+u_4})$, and then the formula of Proposition~\ref{mleoneprop} reduces to $L(u) = \overline{w\beta}_\A(p) = \overline{w\beta}_\A(u_2+u_4,u_3+u_4)$.  Then  \eqref{mlesquare} follows since $1 = u_1 + u_2 + u_3 + u_4$.

The linear forms appearing in the numerators of \eqref{mlesquare} can be computed by
\[
\kbordermatrix{
& m_1 & m_2 & m_3 & m_4\cr
F_1 & 0 & 1 & 0 & 1 \cr
F_2 & 0 &  0 & 1 & 1 \cr
F_3 & 1 &  0 & 1 &  0 \cr
F_4 & 1 &  1 & 0 & 0 } 
\begin{bmatrix} u_1\\ u_2\\ u_3\\ u_4\end{bmatrix}
=
\begin{bmatrix} u_2+u_4\\ u_3+u_4\\ u_1+u_3\\ u_1+u_2\end{bmatrix},
\]
where the row and column labels indicate that the entry in position $(i,j)$ is the lattice distance from lattice point $m_j$ to facet $F_i$.  We will see in Section~\ref{hornsection} that this is no accident.  
\end{example}

The special behavior noted in Example~\ref{square1} generalizes as follows.

\begin{corollary}
\label{corLisTBP}
If $P$ has strict linear precision for weights $w$, then $Y_{\A,w}$ has ML degree one and the ML estimate of $u \in \R^s_{>0}$ is given by the formula
\[
L(u) = \overline{w \beta}_\A(\tau_\A (u)).
\]
\end{corollary}

\begin{proof}
By Definition \ref{slpdef}, any $p \in P$ satisifes 
\[
p = K_w(p) = \frac1{\sum_{m \in \A} {w}_m \hskip1pt \beta_m(p)} \sum_{m \in \A} {w}_m \hskip1pt \beta_m(p)\hskip1pt m.
\]
The formula on the right is the rational function $\tau_\A \circ  \overline{w \beta}_\A$ applied to $p \in P$.  Since $P \subseteq \R^d \subseteq \C^d$ is Zarsiki dense in $\C^d$, it follows that $ \tau_\A \circ  \overline{w \beta}_\A$ is the identity on $\C^d$, hence birational.  Then Proposition~\ref{mleoneprop} implies that $Y_{\A,w}$ has ML degree one, and since the inverse of the identity is the identity, the formula for the ML estimate given in the proposition simplifies to 
 $L(u) = \overline{w \beta}_\A(\tau_\A (u))$.
\end{proof}

\begin{remark}
\label{Kwwbtau}
The above proof shows that 
\[
K_w = \tau_\A \circ  \overline{w \beta}_\A
\]
where $K_w$ is from Definition~\ref{K-definition}.
\end{remark}

\subsection{The Horn Parametrization and the Horn Matrix}

When the ML degree is one, we have the following important theorem of Huh \cite[Thm.\ 1]{huh}.

\begin{theorem}
\label{huhthm}
A variety $Y \subseteq (\C^*)^s$ has ML degree one if and only if there is an $n \times s$ integer matrix $H = (b_{kj})$ with zero column sums and nonzero constants $d_1,\dots,d_s$ such that the  map sending $u = (u_1,\dots,u_s) \in \C^s$ to
\begin{equation}
\label{huhthmpsi}
\mathcal{H}(u) = 
\Big(d_1 \prod_{k=1}^n (b_{k1}u_1+\cdots + b_{ks}u_s)^{b_{k1}},
\dots,d_s \prod_{k=1}^n (b_{k1}u_1+\cdots + b_{ks}u_s)^{b_{ks}}\Big)
\end{equation}
is a dominant rational map of $\C^s \dashrightarrow Y$, i.e., its image is Zariski dense in $Y$.  Furthermore, when $u \in \Z_>^s$ is sufficiently general, the map $\ell_u(x) = x^u$ has a unique critical point $\hat{x} \in Y$ given by \eqref{huhthmpsi}.
\end{theorem}

\begin{proof}
Huh's theorem implies that \eqref{huhthmpsi} is a parametrization $Y$.  The final assertion of Theorem~\ref{huhthm} follows from  the remarks following the statement of \cite[Thm.\ 1]{huh}. 
\end{proof}

The map $\mathcal{H}(u)$ in \eqref{huhthmpsi}  is called the \emph{Horn parametrization} of $Y$, and $H$ is called the \emph{Horn matrix}.

\begin{corollary}
\label{HuhCor}
If $Y_{\A,w} \subseteq (\C^*)^s$ has ML degree one, then for general $u \in \Z^s$, the ML estimate $L(u)$ is given by the Horn parametrization $\mathcal{H}(u)$ from \eqref{huhthmpsi}.
\end{corollary}

\begin{proof} In Section~\ref{mlesection}, we learned that the ML estimate $L(u)$ is a critical point of $\ell_u = x^u$ on $Y_{\A,w}$.  
But this critical point is $\mathcal{H}(u)$ by  Theorem~\ref{huhthm}.  It follows immediately $\mathcal{H}(u) = L(u)$.
\end{proof}



A key feature of \eqref{huhthmpsi} is that \emph{exponents are coefficients}.  Note that  the components of \eqref{huhthmpsi} are rational functions of degree zero in $u$ since the column sums of $H$ are zero. 

\begin{example}
\label{square2}
Consider 
\[
H = \begin{bmatrix} 
\ph0 & \ph1 & \ph0 & \ph1 \\
\ph0 & \ph0 & \ph1 & \ph1 \\
\ph1 &  \ph0 & \ph1 &  \ph0 \\
\ph1 &  \ph1 & \ph0 & \ph0\\
-2 & -2 & -2 & -2
\end{bmatrix}.
\]
This is a Horn matrix since the column sums are zero.  Then
\[ 
H\begin{bmatrix} u_1\\ u_2\\ u_3\\ u_4\end{bmatrix}
=
\begin{bmatrix} u_2+u_4\\ u_3+u_4\\ u_1+u_3\\ u_1+u_2\\-2u_+\end{bmatrix}
\]
shows that when $u_+=1$, the ML estimate \eqref{mlesquare}  given in Example~\ref{square1} comes from the Horn matrix $H$ with constants $(d_1,d_2,d_3,d_4) = (4,4,4,4)$.
\end{example}

\subsection{Minimal Horn Matrices}  In general, different Horn matrices can give the same Horn parametrization.  In order to state our result, we need two definitions.

\begin{definition}
\label{minimalHorndef}
A Horn matrix $H$ is \emph{minimal} if no two rows are linearly dependent.  
\end{definition}

For the next definition, first note that by unique factorization, a collection of $s$ nonzero rational functions $G_1,\dots,G_s \in \C(u_1,\dots,u_s)$ can be written
\begin{equation}
\label{reducedrepresentation}
G_j = c_j \prod_{k=1}^n f_k^{a_{kj}}
\end{equation}
where:
\begin{itemize}
\item $c_j \in \C^*$.
\item $f_1,\dots,f_n \in \C[u_1,\dots,u_s]$ are irreducible and for all $i \ne j$, $f_i$  is not a constant multiple of $f_j$.
\item $a_{kj} \in \Z$ and for each $k = 1,\dots,n$, $a_{kj} \ne 0$ for at least one $j$.
\end{itemize}
Thus the $f_k$ are the irreducible polynomials that appear in the irreducible factorizations of the numerators ($a_{kj} > 0$) and denominators ($a_{kj} < 0$) of the $G_i$ after common factors have been removed.  

\begin{definition}
\label{redrepdef}
When the above three conditions are satistied, we say that \eqref{reducedrepresentation} is a \emph{reduced prepresentation} of $G_1,\dots,G_s$, and we call $(a_{ki})$ the \emph{exponent matrix} of the reduced representation.
\end{definition}

\begin{proposition}
\label{RedHornMatrixProp}
If $\dim(Y) > 0$, a Horn parametrization $\mathcal{H}$ of $Y$  comes from a minimal  Horn matrix which is unique up to permutation of the rows.  Furthermore:
\begin{enumerate}
\item If we write $\mathcal{H} = (G_1,\dots,G_s)$, then the minimal Horn matrix is the exponent matrix of a reduced representation of the $G_i$.
\item If $H$ is another Horn matrix that gives $\mathcal{H}$, then the rows of the minimal Horn matrix are the sums of rows of $H$ lying on the same line through the origin, omitting any zero rows.
\end{enumerate}
\end{proposition}

\begin{proof}
Suppose that $\mathcal{H}$ comes from a $n\times s$ Horn matrix $H = (b_{kj})$ with constants $d_1,\dots,d_s$.  First observe that a zero row of $H$ gives a factor of $0^0 = 1$ (this is the convention used in \cite{huh}) in each product in \eqref{huhthmpsi}.  Hence we may assume that $H$ contains no zero rows.  It follows that
\[
\ell_k(u) = \sum_{j=1}^s b_{kj} u_j
\]
is irreducible in $\C[u_1,\dots,u_s]$ for $k = 1,\dots, n$. 

A minimal Horn matrix is formed from $H$ by taking as rows the sums of $H$'s rows that line on the same line through the origin, discarding any zero rows obtained along the way.  The constants $d_j$ can then be adjusted to give the same $\mathcal{H}$.  

Now assume $H$ is minimal. Consider a product that appears in the Horn parametrization \eqref{huhthmpsi}
\[
d_j \prod_{k=1}^n \ell_k(u)^{b_{kj}}.
\]
Minimal implies that the $\ell_k$ are irreducible and distinct, and for each $k$, at least one exponent $b_{kj}$ is nonzero.  Thus we have a reduced representation with $H$ as the exponent matrix of the representation.  Furthermore, unique factorization implies that $H$ is unique up to a permutation of rows. 
 \end{proof}

\section{Strict Linear Precision and the Horn Matrix}
\label{hornsection}

Recall that $P \subseteq \R^d$ is a $d$-dimensional lattice polytope with facets $F_1,\dots,F_r$ (indexed by $i$) and lattice points $\A = P\cap \Z^d = \{m_1,\dots,m_s\}$ (indexed by $j$).

Before stating our main result, we need some notation.  Set
\[
n_P = \sum_{\rho\in\Sigma(1)} n_\rho = \sum_{i=1}^r n_i,\ \ 
a_P = \sum_{\rho\in\Sigma(1)} a_\rho = \sum_{i=1}^r a_i.
\]
Also recall the lattice distance functions $h_i(p) = \langle p,n_i\rangle + a_i$.  Then define
\[
h_P(p) = \sum_{\rho\in\Sigma(1)} h_\rho(p) = \sum_{i=1}^r h_i(p) = 
\langle p,n_P\rangle + a_P.
\]
Translating $P \subseteq \R^d$ by $m \in \Z^d$ gives a new polytope $P+m$.  Clearly, $n_{P+m} = n_P$, while $a_{P+m} = a_P - \langle m,n_P\rangle$.  Note that when $n_P = 0$, $a_P$ is also independent of translation and $h_P$ is the constant function $h_P(p) = a_P$.   

The following $(r+1)\times s$ matrix $H$ will appear in Theorem~\ref{XML}:
\begin{equation}
\label{Hdef}
H = 
\kbordermatrix{
& m_1 & m_2 & \dots & m_s\cr
F_1  & h_1(m_1) &  h_1(m_2) & \cdots &  h_1(m_s) \cr
F_2 &  h_2(m_1) &  h_2(m_2) & \cdots &   h_2(m_s) \cr
\vdots &  \vdots &  \vdots & \ddots &   \vdots \cr
F_r &   h_r(m_1) &   h_r(m_2) & \cdots &   h_r(m_s)\cr
 & -a_P &  {-}a_P & \cdots & -a_P}.
\end{equation}
Note that the first $r$ rows have non-negative entries since $m_1,\dots,m_s \in P$.
 
\begin{proposition} 
\label{XMLprop}
If $n_P = 0$, then the matrix $H$ of \eqref{Hdef} is a minimal Horn matrix.
\end{proposition}

\begin{proof}
Without assuming $n_P = 0$, we first show that
\begin{align}
\label{firstline}
&\text{no two of the first $r$ rows are multiples of each other, and}\\
\label{secondline}
&\text{none of the first $r$ rows is a multiple of $[1,\dots,1]$.}
\end{align}
For \eqref{firstline}, suppose there are $i \ne k$, both $\le r$, such that the  $i$\textsuperscript{th} row of $H$ equals $\lambda$ times the  $j$\textsuperscript{th} row. This means
\[
h_i(m_j) = \lambda h_k(m_j), \quad j = 1,\dots,s.
\]
Since the $m_j$ are the lattice points of a full dimensional lattice polytope $P$, this implies $h_i = \lambda h_k$.  Thus $h_i$ and $h_k$ define the same facet of $P$, which is impossible since $i \ne k$.
  
For \eqref{secondline}, suppose we have $i \le r$ such that the $i$\textsuperscript{th} row of $H$ is $\lambda$ times $[1,\dots,1]$.  Then
\[
h_i(m_j) = \langle m_j, n_i\rangle + a_i = \lambda,  \quad j = 1,\dots,s.
\]
Since the $m_j$ affinely span $\R^d$, this implies that $h_i$ is constant, which is impossible since $n_i \ne 0$.  

Now assume $n_P = 0$.  As noted above, this implies that $h_P$ is the constant function $h_P(p) = a_P$.  This has two consequences:
\begin{itemize}
\item $a_P = h_P(m_j) = \sum_{i=1}^r h_i(m_j)$, so $H$ is a Horn matrix. 
\item If $p$ is in the interior of $P$, then $h_i(p) > 0$ for all $i$, hence $a_P = h_P(p) > 0$.  
\end{itemize}
To prove minimality, note that rows $i \ne k$ of $H$ with $i,k \le r$ are not linearly dependent by \eqref{firstline}.  Furthermore, since $a_P >0$, row $r+1$ is a nonzero multiple of $[1.\dots,1]$ and hence is not a multiple of any other row by \eqref{secondline}.
\end{proof}

\begin{theorem}[Main Theorem]
\label{XML}
Let $P \subseteq \R^d$ be a $d$-dimensional lattice polytope.  Then the following are equivalent for positive weights $w = (w_j)_{j=1}^s${\rm:}
\begin{enumerate}
\item $P$ has strict linear precision for $w$. 
\item $n_P =0$ and $\beta_w(p)= \sum_{j=1}^s w_j \beta_j(p) = \sum_{j=1}^s w_j \prod_{i=1}^r h_i(p)^{h_i(m_j)}$ is a nonzero constant $c$.
\end{enumerate}
Furthermore, when {\rm(1)} and {\rm(2)} hold, we have:
\begin{itemize}
\item[(a)] $Y_{\A,w} \subseteq (\C^*)^s$ has ML degree one.
\item[(b)] The Horn parametrization $\mathcal{H}$ of $Y_{\A,w}$ comes from the matrix $H$ in \eqref{Hdef} with constants $d_j =  (w_j/c) (-a_P)^{a_P}$ for $1 \le j \le s$.   Furthermore, $H$ is a minimal Horn matrix.
\item[(c)] The ML estimate $L(u)$ for general $u \in \Z^s$ is given by $\mathcal{H}(u)$.
\end{itemize}
Finally, when $X_P$ is smooth, {\rm(1)} and {\rm(2)} are equivalent to 
\begin{enumerate}
\item[(3)] $\mu_{\text{\rm quot}} = \mu_{\text{\rm FS},w}$.
\end{enumerate}
\end{theorem}

\begin{proof}
(1) $\Rightarrow$ (2):  
To begin, we note that (1) and Corollary~\ref{corLisTBP} imply that $Y_{\A,w}$ has ML degree one, which proves (a).  When we combine this with Corollary~\ref{HuhCor}, we get
\[
\mathcal{H}(u) = L(u) = \overline{w\beta}_\A(\tau_\A(u))
\]
for general $u \in \Z^s$.  This proves (c), and since these $u$'s are Zariski dense in $\C^s$, we see that 
\begin{equation}
\label{Hwbtau}
\mathcal{H} = \overline{w\beta}_\A\circ \tau_\A
\end{equation}
as rational maps $\C^s \dashrightarrow Y_{\A,w}$.  The heart of the proof is to compute a reduced representation of the right-hand side of  \eqref{Hwbtau} and use the ``\emph{exponents are coefficients}'' property of the Horn parametrization \eqref{huhthmpsi}.

For the right-hand side of \eqref{Hwbtau}, we set $p = \tau_\A(u) = \sum_{j=1}^s \frac{u_j\,}{u_+}\hskip1pt m_j$.  Then
\begin{equation}
\label{wbtauformula1}
(\overline{w\beta}_\A\circ \tau_\A)(u) = \overline{w\beta}_\A(p) = \frac{1}{\beta_w(p)} \big(w_1\beta_1(p),\dots,w_s\beta_s(p)\big).
\end{equation}
To write this more explicitly, define linear forms $\ell_1(u),\dots,\ell_r(u)$  by
\[
H \begin{bmatrix} u_1\\ u_2\\ \vdots \\ u_s\end{bmatrix} =
\begin{bmatrix} \ell_1(u)\\ \ell_2(u)\\ \vdots \\ \ell_r(u)\\ -a_P\hskip1pt u_+\end{bmatrix},
\]
where $H$ is the matrix from \eqref{Hdef} and $u_+ = u_1+\cdots+u_s$.  Then we compute:
\begin{equation}
\label{hili}
\begin{aligned}
h_i(p) &= \langle p,n_i\rangle + a_i\\ &= \Big\langle \sum_{j=1}^s \frac{u_j\,}{u_+} m_j,n_i\Big\rangle + a_i
= \Big\langle \sum_{j=1}^s \frac{u_j\,}{u_+} m_j,n_i\Big\rangle + a_i\sum_{j=1}^s \frac{u_j\,}{u_+}\\
&= \frac1{u_+} \sum_{j=1}^s (\langle m_j,n_i\rangle + a_i) u_j =  \frac1{u_+} \sum_{j=1}^s h_i(m_j) u_j =\frac{\ell_i(u)}{u_+},
\end{aligned}
\end{equation}
where the last equality follows from the definition of $H$ and $\ell_i$. It follows that
\begin{equation}
\label{himliu}
\beta_j(p) = \prod_{i=1}^r h_i(p)^{h_i(m_j)} = \prod_{i=1}^r \Big(\frac{\ell_i(u)}{u_+}\Big)^{h_i(m_j)} = \frac{
\prod_{i=1}^r \ell_i(u)^{h_i(m_j)}}{u_+^{h_P(m_j)}}
\end{equation}
by the definition of $h_P$.  Setting $h = \max(h_P(m_1),\dots,h_P(m_s))$, we obtain
\begin{equation}
\label{betamformula}
\beta_w(p) = \sum_{j=1}^s w_j\hskip1pt \beta_j(p) = 
\sum_{j=1}^s w_j   \frac{\prod_{i=1}^r \ell_i(u)^{h_i(m_j)}}{u_+^{h_P(m_j)}} =
\frac{\tilde\beta(u)}{u_+^{h}},
\end{equation}
where 
\[
\tilde\beta(u) = \sum_{j=1}^s w_j\hskip1pt u_+^{h-h_P(m_j)}\hskip1pt{\textstyle\prod_{i=1}^r} \ell_i(u)^{h_i(m_j)}.
\]
Note that $\tilde\beta(u)$ is homogeneous of degree $h$ in $u_1,\dots,u_s$.  In this notation, the $j$\textsuperscript{th} component of \eqref{wbtauformula1} is
\begin{equation}
\label{slppsi}
\frac{w_j \beta_j(p)}{\beta_w(p)} = \frac{w_j\hskip1pt u_+^{h-h_P(m_j)}\prod_{i=1}^r \ell_i(u)^{h_i(m_j)}}{\tilde\beta(u)}.
\end{equation}

We next observe that $\ell_1(u),\dots,\ell_r(u)$ and $u_+$ are irreducible and distinct, i.e., none is a multiple of any other.  This follows immediately from \eqref{firstline} and \eqref{secondline}.  We further claim that none of $\ell_1(u),\dots,\ell_r(u)$ divide $\tilde\beta(u)$.   Fix $1 \le i \le r$ and pick $p \in P$ on the facet defined by $h_i = 0$, so that $h_i(p) = 0$.  We can find $u \in \Delta$ such that $p = \tau_\A(u)$.   Since $u_+ = 1$, \eqref{hili} and \eqref{betamformula} imply
\[
\ell_i(u) = h_i(p) = 0 \ \text{and}\  \tilde\beta(u) = \beta_w(p) > 0,
\]
where we use the fact that $\beta_w$ is positive on $P$.  It follows that $\ell_i$ cannot divide $\tilde\beta$. 

At this point, we don't know how $u_+$ relates to $\tilde\beta(u)$, so we write the unique factorization of $\tilde\beta(u)$ as
\begin{equation}
\label{tbfactors}
\tilde\beta(u) = c\hskip1pt u_+^a f_1^{a_1}(u)\cdots f_\nu^{a_\nu}(u)
\end{equation}
where $c \in \C^*$, $a \ge 0$, $a_1,\dots,a_\nu > 0$ and $u_+,\ell_1(u),\dots,\ell_s(u),f_1(u),\dots,f_\nu(u)$ are distinct irreducibles.  Combining this with \eqref{slppsi}, we see that
\begin{equation}
\label{redwbtau}
\frac{w_j \beta_j(p)}{\beta_w(p)} = w_j\,u_+^{h-h_P(m_j)-a}\frac{\prod_{i=1}^r \ell_i(u)^{h_i(m_j)}}{c\hskip1pt f_1^{a_1}(u)\cdots f_\nu^{a_\nu}(u)}, \quad j=1,\dots,s
\end{equation}
is a reduced representation of the components of $\overline{w\beta}_\A\circ \tau_\A$.  By \eqref{Hwbtau}, this gives the reduced representation of $\mathcal{H}(u)$, so that by Proposition~\ref{RedHornMatrixProp}, the minimal Horn matrix of $\mathcal{H}(u)$ is the exponent matrix of \eqref{redwbtau}. 

Now the magic happens.  Consider $f_1$ on right-hand side of \eqref{redwbtau}.  Since it has the same exponent $a_1$ for $j=1,\dots,s$, the exponent matrix has the row $[-a_1,\dots,-a_1]$, and since exponents are coefficients, this means that
\[
f_1 = C(-a_1 u_1 - \cdots  -a_1 u_s) = -Ca_1 u_+
\]
 for some constant $C$.  But in \eqref{tbfactors}, $f_1$ is not a multiple of $u_+$.  This means that $f_1,\dots,f_\nu$ don't exist! Hence \eqref{tbfactors} reduces to
\[
\tilde\beta(u) = c\hskip1pt u_+^h.
\]
The exponent is $a = h$ since $\tilde\beta(u)$ is homogeneous of degree $h$.  Then \eqref{betamformula} becomes
\[
\beta_w(p) = \frac{\tilde\beta(u)}{u_+^{h}} = \frac{c\hskip1pt u_+^h}{u_+^{h}} = c.
\]
This proves the $\beta_w(p)$ is the constant polynomial.  Furthermore, \eqref{redwbtau} simplifies to
\[
\frac{w_j \beta_j(p)}{\beta_w(p)} = (w_j/c)\hskip1pt u_+^{-h_P(m_j)}\prod_{i=1}^r \ell_i(u)^{h_i(m_j)}, \quad j=1,\dots,s
\]
In the exponent matrix, $u_+$ gives the row $[-h_P(m_1),\dots,-h_P(m_s)]$.  But exponents are coefficients, so we must have $-h_P(m_1) =\dots =-h_P(m_s)$.  Since the $m_j$ affinely span $\R^d$, the formula $h_P(p) = \langle p,n_P\rangle + a_P$ must be constant, which forces $n_P = \sum_{i=1}^r n_i = 0$. 

This completes the proof of (1) $\Rightarrow$ (2), and along the way we have also proved (a) and (c).  It remains to prove (b).  The previous paragraph implies $h_P(m_j) = a_P$ for all $j$, so the above formula for the components of $\overline{w\beta}_\A\circ\tau_\A$ becomes
\begin{equation}
\label{finalsimp}
\begin{aligned}
\frac{w_j \beta_j(p)}{\beta_w(p)} &= (w_j/c)\hskip1pt u_+^{-a_P}\prod_{i=1}^r \ell_i(u)^{h_i(m_j)}\\ &= (w_j/c)(-a_P)^{a_P}\hskip1pt (-a_P\hskip1pt u_+)^{-a_P}\prod_{i=1}^r \ell_i(u)^{h_i(m_j)}.
\end{aligned}
\end{equation}
By \eqref{Hwbtau}, these are the components of $\mathcal{H}(u)$.  Then the second line of \eqref{finalsimp} shows that $H$ is indeed a Horn matrix for 
$\mathcal{H}(u)$, and since $n_P = \sum_{i=1}^r n_i = 0$, $H$ is minimal by Proposition~\ref{XMLprop}.  This completes the proof of (c).

(2) $\Rightarrow$ (1): $n_P = 0$ implies $h_P(m_j) = a_P$ for all $j$, and $\beta_w(p) = c$ and  \eqref{betamformula} imply $\tilde\beta(u) = c\hskip1pt u_+^h$.  Then \eqref{slppsi} simplifes to 
\eqref{finalsimp}.  This shows that $\overline{w\beta}_\A\circ\tau_\A$ is a Horn parametrization of $Y_{\A,w}$ for $H$ from \eqref{Hdef}, which is a minimal Horn matrix by Proposition~\ref{XMLprop}.

Corollary~\ref{HuhCor} implies that $\overline{w\beta}_\A\circ\tau_\A(u)$ is the ML estimate $L(u)$ for general $u \in \Z_{>0}^s$.  Since $\tau_\A(L(u)) = \tau_\A(u)$ by Proposition~\ref{mleprop}, we see that
\[
(\tau_\A \circ \overline{w\beta}_\A\circ\tau_\A)(u) = \tau_\A(L(u)) = \tau_\A(u)
\]
for general $u \in \Z_{>0}^s$.  These $u$'s are Zariski dense in $\C^s$, so 
\[
\tau_\A \circ \overline{w\beta}_\A\circ\tau_\A = \tau_\A
\]
as rational functions on $\C^s$.  Since $\tau_\A$ is dominating, it follows that $\tau_\A \circ \overline{w\beta}_\A$ is the identity as a rational function on $\C^d$.  Restricting to $P$, where everything is defined, we conclude that
\[
p = (\tau_\A \circ \overline{w\beta}_\A)(p) = \frac{1}{\sum_{j=1}^s w_j \hskip1pt \beta_j(p)} \sum_{j=1}^s w_j \hskip1pt \beta_j(p) \hskip1pt m_j
\]
for all $p \in P$.  This proves that $P$ has strict linear precision for the weights $w$.

Finally, we need to consider (3) when $X_P$ is smooth.  By Theorem~\ref{momentprecision}, we know that (3) holds for weights $w = (w_j)_{j=1}^s$ if and only if (1) holds for weights $\mathbf{w} = (2^{\langle m_j,n_P\rangle} w_j)_{j=1}^s$.  

If (1) holds for $w$, then so does (2), hence $n_P = 0$, so that $\mathbf{w} = w$.  Then (1) holds for $\mathbf{w}$ and consequently (3) holds for $w$.  Conversely, if (3) holds for $w$, then (1) holds for $\mathbf{w}$.  Hence (2) holds, which implies $n_P = 0$, so that $\mathbf{w} = w$.  Thus (1) holds for $w$ and we are done.
\end{proof}

Also observe that Theorem~\ref{strictequivalence} is an immediate consequence of Theorem~\ref{XML}, and
Examples~\ref{square1} and \ref{square2} give a concrete application of the theorem.

\section{Rational Linear Precision}
\label{rationalsection}
Theorem~\ref{XML} shows that strict linear precision implies ML degree one.  It is important to note that the converse is not true.  
In this section, we discuss a more general notion of linear precision and give two extended examples of what can happen. As usual, $P \subseteq \Z^d$ is a $d$-dimensional lattice polytope with lattice points $\A = P \cap\Z^d = \{m_1,\dots,m_s\}$.  

\begin{definition} (Garcia-Puente and Sottile \cite{SG})
\label{rlpdef}
$P \subseteq \Z^d$ has \emph{rational linear precision} for positive weights $w = (w_j)_{j=1}^s$ if there are rational functions $\hat\beta_1,\dots,\hat\beta_s$ on $\C^d$ satisfying:
\begin{enumerate}
\item $\sum_{i=1}^s \hat\beta_j = 1$ as rational functions on $\C^d$.
\item $\hat\beta_1,\dots,\hat\beta_s$ define a rational parametrization
\[
\hat\beta :\C^d \dashrightarrow X_{\A,w} \subseteq \PP^{s-1},\quad \hat\beta(t) = \big(\hat\beta_1(t),\dots,\hat\beta_s(t)\big).
\]
\item For every $p \in P \subseteq \C^d$, $\hat\beta_j(p)$ is defined and is a nonnegative real number.
\item $\sum_{j=1}^s \hat\beta_j(p)\hskip1pt m_j = p$ for all $p \in P$.
\end{enumerate}
\end{definition}


\begin{example}
\label{strictrational}
Any polytope with strict linear precision for weights $(w_j)_{j=1}^s$ has rational linear precision.  To see why, observe that the toric blending functions $w_j\hskip1pt \beta_j(t)/\beta_w(t)$ parametrize $X_{\A,w}$ by Proposition~\ref{XYparam}.  Since strict linear precision implies
\[
p = \frac1{\beta_w(p)}\sum_{i=1}^s w_j\hskip1pt \beta_j(p) \hskip1pt  m_j
\]
for $p \in P$, the conditions of Definition~\ref{rlpdef} are satisfied.
\end{example}

Later in this section, we will encounter two polytopes that have rational but not strict linear precision.  For now, we give a useful result about rational linear precision which combines several results of \cite{SG}.  

\begin{proposition}
\label{SGprop41}
The polytope $P$ has rational linear precision for positive weights $(w_j)_{j=1}^s$ if and only if $Y_{\A,w} \subseteq (\C^*)^s$ has ML degree one.
\end{proposition}

\begin{proof}
If $\hat\beta_1,\dots,\hat\beta_s$ satisfy Definition~\ref{rlpdef}, then $t \mapsto \sum_{j=1}^s\hat\beta_j(t)\hskip1pt m_j$ is birational since it is the identity on $P$.  By Proposition~\ref{mleoneprop},  $Y_{\A,w}$ has ML degree one.

For the converse, let $Y_{\A,w}$ have ML degree one.  By Corollary~\ref{mlonecor}, $\tau =\tau_{\A}\raisebox{-2.5pt}{$|$}{\raisebox{-3pt}{\scriptsize$Y_{\A,w}$}}$ is a birational map $\tau:Y_{\A,w} \dashrightarrow \C^d$.  If we write its inverse $\tau^{-1} : \C^d \dashrightarrow Y_{\A,w}$ as
\[
\tau^{-1}(t) = \hat\beta(t) = \big(\hat\beta_1(t),\dots,\hat\beta_s(t)\big) \in Y_{\A,w} \subseteq (\C^*)^s,
\]
then $\hat\beta_1,\dots,\hat\beta_s$ clearly satisfy (1) and (4) of Definition~\ref{rlpdef}.  Also, composing with the obvious map $Y _{\A,w} \to X_{\A,w} \subseteq \PP^{s-1}$, we see that $\hat\beta$ satisfies (2) of  Definition~\ref{rlpdef}. 

In proof of Proposition~\ref{XYparam}, we noted that the toric blending parametrization $\C^d \dashrightarrow X_{\A,w}$ maps $P$ to $X_{\A,w}\cap \R_{\ge0}$.  Krasauskas \cite[Thm.\ 25]{krasauskas} proved that the toric blending functions $w_j\hskip1pt\beta_j/\beta_w$ give a homeomorphism
\[
p \in P \longmapsto \sum_{j=1}^s \frac{w_j\hskip1pt\beta_j(p)}{\beta_w(p)}\hskip1pt m_j \in P.
\]
This implies that $\tau_\A$ maps  $X_{\A,w}\cap \R^s_{\ge0}$ homeomorphically onto $P$, so that its inverse
is a homeomorphism from $P$ onto  $X_{\A,w}\cap \R_{\ge0}$.  It follows that the $\hat\beta_j$ satisfy (3) of Definition~\ref{rlpdef}. 
\end{proof}

Here one situation where rational linear precision is easy to understand.

\begin{proposition}
\label{rlpformula}
Suppose that
\begin{equation}
\label{monbirational}
(\tau_\A\circ \overline{w\chi}_\A)(t) =  \frac1{\sum_{j=1}^s w_j\hskip1pt t^{m_j}} \sum_{j=1}^s w_j\hskip1pt t^{m_j}\hskip1pt m_j
\end{equation} 
defines a birational map from $\C^d\dashrightarrow \C^d$.  Then:
\begin{enumerate}
\item $P$ has rational linear precision with weights $w = (w_j)_{j=1}^s$.
\item Let $\varphi$ denote the inverse of \eqref{monbirational}.  Then 
\[
\hat\beta(t) = (\overline{w\chi}_\A\circ\varphi)(t)= \Big(\frac{w_1\hskip1pt \varphi(t)^{m_1}}{\sum_{j=1}^s w_j\hskip1pt \varphi(t)^{m_j}} ,\dots,\frac{w_s\hskip1pt \varphi(t)^{m_s}}{\sum_{j=1}^s w_j\hskip1pt \varphi(t)^{m_j}} \Big) \in Y_{\A,w}
\] 
satisfies rational linear precision from Definition~\ref{rlpdef}.
\item $Y_{\A,w}$ has ML degree one and the Horn parametrization of $Y_{\A,w}$ is given by
\[
u = (u_1,\dots,u_s)  \longmapsto \hat\beta(p), \quad p = \sum_{j=1}^s \frac{u_j\,}{u_+},\ \  u_+ = \sum_{j=1}^s u_j.
\]
\end{enumerate}
\end{proposition}

\begin{proof}
The map  \eqref{monbirational} is the composition of $\tau = \tau_{\A}\raisebox{-2.5pt}{$|$}{\raisebox{-3pt}{\scriptsize$Y_{\A,w}$}}$ with the monomial parametrization of $Y_{\A,w}$ from Proposition~\ref{mleoneprop}.  This proposition implies that $Y_{\A,w}$ has ML degree one and shows that the formula of (3) is the ML estimate when $u \in \R_{>0}^s$.  By Corollary~\ref{HuhCor}, this is the Horn parametrization.  This proves (3), and then (1) follows by Proposition~\ref{SGprop41}.  

For (2), the proof of Proposition~\ref{SGprop41} shows that $\tau^{-1}$ gives a parametrization of $Y_{\A,w}$ that satisfies Definition~\ref{rlpdef}.  Then we are done since the definition of $\varphi$ easily implies that $\hat\beta = \tau^{-1}$.  
\end{proof}

In practice, Proposition~\ref{rlpformula} is a useful method for finding the parametrization that gives rational linear precision.  We will give two examples to illustrate this.  The proposition also gives the Horn parametrization.  In the case of strict linear precision, we can describe the Horn matrix in advance by Theorem~\ref{XML}.  We are not aware of any such description for rational linear precision, though our examples suggest that the Horn matrices have a very interesting structure in this case.

\subsection{The Trapezoid}
\label{subsection-trapezoid}
The trapezoid 
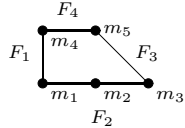
\[
\begin{picture}(62,50)
\put(15,15){\circle*{3.5}}
\put(35,15){\circle*{3.5}}
\put(55,15){\circle*{3.5}}
\put(15,35){\circle*{3.5}}
\put(35,35){\circle*{3.5}}
\put(18,9){\scriptsize $m_1$}
\put(38,9){\scriptsize $m_2$}
\put(58,9){\scriptsize $m_3$}
\put(18,29){\scriptsize $m_4$}
\put(38,34){\scriptsize $m_5$}
\put(15,15){\line(1,0){40}}
\put(15,15){\line(0,1){20}}
\put(35,35){\line(1,-1){20}}
\put(15,35){\line(1,0){20}}
\put(2,24.5){\scriptsize $F_1$}
\put(33,0){\scriptsize $F_2$}
\put(20,41){\scriptsize $F_4$}
\put(50,24.5){\scriptsize $F_3$}
\end{picture}
\]
has lattice points $m_1 = (0,0),\ m_2 = (1,0),\ m_3 = (2,0),\ m_4 = (0,1),\ m_5 = (1,1)$.  It is known (see \cite[Example 1.6]{GRS}) that the trapezoid has rational linear precision for the weights $w = (1, 2,1,1,1)$.  On the other hand, 
\[
n_P = n_1+n_2+n_3+n_4 = (1,0)+(0,1)+(-1,-1)+(0,-1) = (0,-1) \ne (0,0),
\]
so that the trapezoid does not have strict linear precision for any choice of weights by Theorem~\ref{XML}.

For $(s,t) \in (\C^*)^2$, we define 
\[
S(s,t) = 1 + 2s + s^2 + t + st = (1+s)(1+s+t).
\]
Then the monomial parametrization of $Y_{\A,w,}$ is given by
\begin{equation}
\label{trap1}
\overline{w\chi}_\A(s,t) = \frac1{S(s,t)}\big(1, 2s, s^2, t, st\big).
\end{equation}
Composing this with $\tau_\A : Y_{\A,w} \dashrightarrow \C^2$, we obtain 
\begin{equation}
\label{trap1.5}
\begin{aligned}
(\tau_\A\circ \overline{w\chi}_\A)(s,t) &= 
\frac1{S(s,t)}\big(1(0,0)+ 2s(1,0) + s^2(2,0) + t(0,1) + st(1,1)\big)\\
&= \Big(\frac{s(2 + 2 s + t)}{(1 + s) (1 + s + t)}, \frac{t}{1 + s + t}\Big).
\end{aligned}
\end{equation}
This map is birational with inverse
\[
\varphi(s,t) = \Big(\frac{s}{2 - s - t},\frac{t(2-t)}{(1-t)(2-s-t)}\Big),
\]
Proposition \ref{rlpformula} implies rational linear precision. 
When we substitute this into \eqref{trap1}, we get the parametrization that sends $p = (s,t)$ to 
\begin{equation}
\label{trap2}
\begin{aligned}
&\Big(\frac{(1-t)(2-s-t)^2}{(2-t)^2},\frac{2s(1-t)(2-s-t)}{(2-t)^2},\frac{s^2(1-t)}{(2-t)^2},\frac{t(2-s-t)}{2-t},\frac{st}{2-t}\Big) =\\
&\Big(\frac{h_3(p)^2h_4(p)}{(2-t)^2},\frac{2h_1(p)h_3(p)h_4(p)}{(2-t)^2},\frac{h_1(p)^2 h_4(p)}{(2-t)^2},\frac{h_2(p)h_3(p)}{2-t},\frac{h_1(p)h_2(p)}{2-t}\Big)
\end{aligned}
\end{equation}
where $h_1(p) = s,\ h_2(p) = t,\ h_3(p) = 2-s-t,\ h_4(p) = 1-t$. By Proposition~\ref{rlpformula}, this parametrization has rational linear precision.

In contrast, the toric blending parametrization $\overline{w\beta}_\A$ of $Y_{\A,w}$ sends $p = (s,t)$ to
\begin{equation}
\label{trap3}
\Big(\frac{h_3(p)^2h_4(p)}{\beta_w(p)}, \frac{2h_1(p)h_3(p)h_4(p)}{\beta_w(p)}, \frac{h_1(p)^2 h_4(p)}{\beta_w(p)}, \frac{h_2(p)h_3(p)}{\beta_w(p)}, \frac{h_1(p)h_2(p)}{\beta_w(p)}\Big),
\end{equation}
where $\beta_w(p) = \beta_w(s,t) = (2-t)(2-2t-t^2)$.  Comparing this to \eqref{trap2}, we see that the numerators are the same, but the denominators differ.  

We next use \eqref{trap2} to compute the Horn parametrization.  Given $u = (u_1,\dots,u_5)$ and $1 = u_1 + \dotsm + u_5$, we define $p$ as usual via
\[
p = \sum_{j=1}^5 u_j \hskip1pt m_j = \Big({u_2+2u_3+u_5},{u_4+u_5}\Big).
\]
By Proposition~\ref{rlpformula}, we get the Horn parametrization of $Y_{\A,w}$ by evaluating \eqref{trap2} at the point $p$ just defined.  After some algebra, the result is
\[
\big(L_1(u), L_2(u), L_3(u), L_4(u), L_5(u)\big),
\]
where
\begin{equation}
\label{trapML}
\begin{aligned}
L_1(u) &= \frac{  (2 u_1 + u_2 + u_4)^2(u_1 + u_2 + u_3)}{(u_1 + u_2 + u_3 + u_4 + u_5)(2 u_1 + 2 u_2 + 2 u_3 + u_4 + u_5)^2 }\\
L_2(u) &= 2\frac{ (u_2 + 2 u_3 + u_5) (2 u_1 + u_2 + u_4)(u_1 + u_2 + u_3)}{ (u_1 + u_2 + u_3 + u_4 + u_5)(2 u_1 + 2 u_2 + 2 u_3 + u_4 + u_5)^2}\\
L_3(u) &= \frac{(u_2 + 2 u_3 + u_5)^2 (u_1 + u_2 + u_3)}{(u_1 + u_2 + u_3 + u_4 + u_5)(2 u_1 + 2 u_2 + 2 u_3 + u_4 + u_5)^2}\\
L_4(u) &= \frac{(u_4 + u_5)(2 u_1 + u_2 + u_4)}{(u_1 + u_2 + u_3 + u_4 + u_5)(2 u_1 + 2 u_2 + 2 u_3 + u_4 + u_5)}\\
L_5(u) &= \frac{(u_2 + 2 u_3 + u_5)(u_4 + u_5)}{(u_1 + u_2 + u_3 + u_4 + u_5)(2 u_1 + 2 u_2 + 2 u_3 + u_4 + u_5)}.
\end{aligned}
\end{equation}
The reader can check that the linear forms in these formulas come from the minimal Horn matrix
\begin{equation}
\label{trapHorn}
H = 
\kbordermatrix{& m_1 & m_2 & m_3 & m_4 & m_5\cr
F_1 &\phantom{-}0 & \phantom{-} 1 & \phantom{-} 2 & \phantom{-} 0 & \phantom{-} 1\cr 
F_2 &\phantom{-} 0 & \phantom{-} 0 & \phantom{-} 0 & \phantom{-} 1 &  \phantom{-} 1\cr
F_3 &\phantom{-} 2 & \phantom{-} 1 & \phantom{-} 0 & \phantom{-} 1 &  \phantom{-} 0\cr
F_4 & \phantom{-} 1 & \phantom{-} 1 & \phantom{-} 1 & \phantom{-} 0 &  \phantom{-} 0\cr
?_1 &{-}1 &{-}1 & {-}1 & {-}1 & {-}1\cr
?_2 &{-}2 &{-}2 & {-}2 & {-}1 & {-}1}.
\end{equation}
The first four rows of $H$ fit the pattern observed in Theorem~\ref{XML}:\ they record the lattice distances from the lattice points to the facets $F_1,F_2,F_3,F_4$.  The last two rows of $H$, marked $?_1$ and $?_2$, are more mysterious.  Here is one possible explanation:

\begin{itemize}
\item Row $?_1$ is $-(\text{row }F_2 + \text{row }F_4)$, and row $?_2$ is $-(\text{row }F_1 + \text{row }F_3)$.
\item The normal fan of the trapezoid has two primitive collections, $\{n_2,n_4\}$ and $\{n_1,n_3\}$, where $n_i$ is the inward normal corresponding to $F_i$.
\end{itemize}
So primitive collections explain rows $?_1$ and $?_2$ in this example.  See \cite[\S6.4]{cls} for more on primitive collections.

\subsection{Moment Maps of the Trapezoid}
This paper began with the moment maps $\mu_\text{quot}$ and $\mu_{\text{FS},w}$.  Theorem~\ref{XML} tells us what happens in the presence of strict linear precision.  So what about the trapezoid, which in a sense is the next simplest case?  Recall that we are using the weights $w = (1,2,1,1,1)$.

In Proposition~\ref{wtstructure}, we learned that
\[
\mu_{\text{FS},w} = \tau_\A\circ \overline{w\chi}_\A\circ |{\bullet}|^2.
\]
We computed $\tau_\A\circ \overline{w\chi}_\A$ in \eqref{trap1.5}, so that we get a totally explicit formula for $\mu_{\text{FS},w}$ as a rational expression in $|x|^2$.  

For the weights $w = (1,2,1,1,1)$, Proposition~\ref{Km=tx-proposition} and  $n_P = (0,-1)$  imply that
\[
K_w \circ \mu_\text{quot} = \mu_{\text{FS},w'}
\]
for $w' = (1,2,1,2,2)$.  We know a formula for $\mu_{\text{FS},w'}$ from Proposition~\ref{wtstructure}, and in Remark~\ref{Kwwbtau}, we noted that $K_w = \tau_\A\circ \overline{w\beta}_\A$.  We computed $\overline{w\beta}_\A$ in \eqref{trap3}, which makes it easy to show that 
\begin{equation}
\label{trap4}
K_w(s,t) = (\tau_\A\circ \overline{w\beta}_\A)(s,t) = 
\Big(\frac{s(4-5t+2t^2)}{(2-t)(2-2t+t^2)},\frac{t}{2-2t+t^2}\Big).
\end{equation}
In contrast to $\tau_\A\circ \overline{w\chi}_\A$, this map is \emph{not} birational -- it has degree $2$.  Hence
\[
\mu_\text{quot} = K_w^{-1} \circ \mu_{\text{FS},w'}
\]
is an \emph{algebraic function} of $|x|^2$ (square roots appear).  There appears to be no nice formula for the quotient moment map in the presence of rational linear precision.  

\begin{remark}
By Definition~\ref{slpdef}, strict linear precision holds when  $K_w$ is the identity on $P$.  This implies that $\tau_\A\circ \overline{w\beta}_\A$ is the identity as a rational map.  But for the trapezoid, we see that $K_w= \tau_\A\circ \overline{w\beta}_\A$ is far from the identity since it has degree $2$.

The upshot is for the trapezoid, having rational linear precision does not help us understand the quotient moment map.  Strict linear precision is somehow the sweet spot where everything comes together.
\end{remark}

\subsection{A Decomposable Graphical Model} In statistics, the simplest nontrivial decomposable graphical model is 
\begin{equation}
\label{graph}
\begin{array}{c}\begin{picture}(70,20)
\thicklines
\put(2,12){\small 1}
\put(32,12){\small 2}
\put(62,12){\small 3}
\put(5,5){\line(1,0){60}}
\put(5,5){\circle*{4.5}}
\put(35,5){\circle*{4.5}}
\put(65,5){\circle*{4.5}}
\end{picture}
\end{array}
\end{equation}
As explained in \cite{GMS}, we have coordinates $p_{000},\dots,p_{111}$ of $\PP^7$, and the graph \eqref{graph} gives the matrix of columns of the  $t$-monomial parametrization exponents:
\begin{equation}
\label{graphmatrix}
\kbordermatrix{~ & 
p_{000} & p_{001} & p_{010} & p_{011} &  p_{100} & p_{101} & p_{110} & p_{111}\cr
t_1  & 1 & 1 & 0 & 0 & 0 & 0 &0 & 0 \cr
t_2  & 0 & 0 & 1 & 1 & 0 & 0 & 0 & 0\cr
t_3  & 0 & 0 & 0 & 0 & 1 & 1 & 0 & 0\cr
t_4  & 0 & 0 & 0 & 0 & 0 & 0 & 1 & 1\cr
t_5  & 1 & 0 & 0 & 0 & 1 & 0 &0 & 0 \cr
t_6  & 0 & 1 & 0 & 0 & 0 & 1 & 0 & 0\cr
t_7  & 0 & 0 & 1 & 0 & 0 & 0 & 1 & 0\cr
t_8  & 0 & 0 & 0 & 1 & 0 & 0 & 0 & 1}.
\end{equation}
The passage from the graph \eqref{graph} to this matrix is discussed in \cite{GMS}.

Using weights $w = (1,\dots,1)$, the columns $m_1,\dots,m_8$ of \eqref{graphmatrix} are exponent vectors that give the parametrization
\[
t = (t_1,\dots,t_8) \longmapsto (t^{m_1},\dots,t^{m_8}) = [t_1t_5, t_1t_6, t_2t_7,t_2t_8, t_3t_5, t_3t_6, t_4t_7, t_4t_8] \in \PP^7.
\]
As noted in \cite{GMS}, the image of this map is the toric variety defined by
\begin{align*}
&p_{001} p_{100} = (t_1t_6)(t_3t_5) = (t_1t_5)(t_3t_6)  = p_{000}p_{101}\\
&p_{011} p_{110} = (t_2t_8)(t_4t_7) = (t_2t_7)(t_4t_8)  = p_{010}p_{111},
\end{align*}
corresponding to $m_2+m_5 = m_1+m_6$ and $m_4+m_7 = m_3+m_8$.   Since the toric variety lives in $\PP^7$ and is defined by two equations, it has dimension $5$.

It is well known that decomposable graphical models have ML degree one \cite{lauritzen}.  In the case of  \eqref{graph}, we will show later in the section that the estimate for $(u_1,\dots,u_8)$ is given by
{\small\begin{align*}
\Big(&\frac{(u_1+u_2)(u_1+u_5)}{(u_1\hskip-.5pt +\hskip-.5pt u_2\hskip-.5pt +\hskip-.5pt u_5\hskip-.5pt +\hskip-.5pt u_6)}, \hskip-1pt\frac{(u_1+u_2)(u_2+u_6)}{(u_1\hskip-.5pt +\hskip-.5pt u_2\hskip-.5pt +\hskip-.5pt u_5\hskip-.5pt +\hskip-.5pt u_6)}, \hskip-1pt
\frac{(u_3+u_4)(u_3+u_7)}{(u_3\hskip-.5pt +\hskip-.5pt u_4\hskip-.5pt +\hskip-.5pt u_7\hskip-.5pt +\hskip-.5pt u_8)}, \hskip-1pt
\frac{(u_3+u_4)(u_4+u_8)}{(u_3\hskip-.5pt +\hskip-.5pt u_4\hskip-.5pt +\hskip-.5pt u_7\hskip-.5pt +\hskip-.5pt u_8)},\\
&\frac{(u_1+u_5)(u_5+u_6)}{(u_1\hskip-.5pt +\hskip-.5pt u_2\hskip-.5pt +\hskip-.5pt u_5\hskip-.5pt +\hskip-.5pt u_6)}, \hskip-1pt
\frac{(u_2+u_6)(u_5+u_6)}{(u_1\hskip-.5pt +\hskip-.5pt u_2\hskip-.5pt +\hskip-.5pt u_5\hskip-.5pt +\hskip-.5pt u_6)}, \hskip-1pt
\frac{(u_3+u_7)(u_7+u_8)}{(u_3\hskip-.5pt +\hskip-.5pt u_4\hskip-.5pt +\hskip-.5pt u_7\hskip-.5pt +\hskip-.5pt u_8)}, \hskip-1pt
\frac{(u_4+u_8)(u_7+u_8)}{(u_3\hskip-.5pt +\hskip-.5pt u_4\hskip-.5pt +\hskip-.5pt u_7\hskip-.5pt +\hskip-.5pt u_8)}\hskip-.5pt \Big)
\end{align*}}
\hskip-3pt when $1 = u_1+\cdots+u_8$.  From this, it is easy to see that the minimal Horn matrix for this graphical model given by the $11 \times 8$ matrix
\begin{equation}
\label{graphHorn}
H = \kbordermatrix{~ & 
m_1 & m_2 & m_3 & m_4 &  m_5 & m_6 & m_7 & m_8\cr
F_1 &\ph 1 &\ph 1 &\ph 0 &\ph 0 &\ph 0 &\ph 0 &\ph 0 &\ph 0\cr
F_2 &\ph 0 &\ph 0 &\ph 1 &\ph 1 &\ph 0 &\ph 0 &\ph 0 &\ph 0\cr
F_3 &\ph 0 &\ph 0 &\ph 0 &\ph 0 &\ph 1 &\ph 1 &\ph 0 &\ph 0\cr
F_4 &\ph 0 &\ph 0 &\ph 0 &\ph 0 &\ph 0 &\ph 0 &\ph 1 &\ph 1\cr
F_5 &\ph 1 &\ph 0 &\ph 0 &\ph 0 &\ph 1 &\ph 0 &\ph 0 &\ph 0\cr
F_6 &\ph 0 &\ph 1 &\ph 0 &\ph 0 &\ph 0 &\ph 1 &\ph 0 &\ph 0\cr
F_7 &\ph 0 &\ph 0 &\ph 1 &\ph 0 &\ph 0 &\ph 0 &\ph 1 &\ph 0\cr
F_8 &\ph 0 &\ph 0 &\ph 0 &\ph 1 &\ph 0 &\ph 0 &\ph 0 &\ph 1\cr
?_1 & -1 & -1 & -1 & -1 & -1 & -1 & -1 & -1\cr
?_2 & -1 & -1 &\ph 0 &\ph 0 & -1 & -1 &\ph 0 &\ph 0\cr
?_3 &\ph 0 &\ph 0 & -1 & -1 &\ph 0 &\ph 0 & -1 & -1}.
\end{equation}
For time being, ignore the row labels.  Right now, the interesting thing to note is that the first eight rows (the nonnegative ones) are precisely the matrix \eqref{graphmatrix}.

To see how this example fits into our framework, we need to understand the $5$-dimensional polytope
\[
\mathrm{Conv}(m_1,\dots,m_8) \subseteq \R^8
\]
(remember that the toric variety has dimension $5$).  The polytopes arising from decomposable graphical models have been studied in \cite{seth}, where they are shown to be compressed  \cite[Cor.\ 5.7]{seth}.  

We will use the following model of $P$:
\[
P = \mathrm{Conv}(0, e_1,e_2,e_1+e_2, e_5, e_3+e_5, e_4+e_5, e_3+e_4+e_5) \subseteq \R^5.
\]
Write points in $\R^5$ as $p = (t_1,\dots,t_5)$.  Then $\{t_5 = 0\}$ and $\{t_5 = 1\}$ are supporting hyperplanes of $P$ that cut out faces
\begin{align*}
Q_1 &= P\cap \{t_5=0\} = \mathrm{Conv}(0, e_1,e_2,e_1+e_2)\\
Q_2 &= P\cap \{t_5=1\} =  \mathrm{Conv}(e_5, e_3+e_5, e_4+e_5, e_3+e_4+e_5),
\end{align*}
each of which is a copy of the unit square.  Here are some properties of this polytope.

\begin{proposition}
\label{Pproperties}
\
\begin{enumerate}
\item $P$ is the join $Q_1 \star Q_2$.
\item $P$ has eight lattice points 
\begin{align*}
&m_1 = 0,\ m_2 = e_1,\ m_5 = e_2,\ m_6 = e_1+e_2,\\
&m_3 = e_5,\ m_4 = e_3+e_5,\ m_7 = e_4+e_5,\ m_8 = e_3+e_4+e_5.
\end{align*}
\item $P$ has eight facets $F_1,\dots,F_8$ with facet presentation
\[
P = \{ p \in \R^5 \mid h_i(p) \ge 0,\ i = 1,\dots,8\},
\]
where $h_i(p) = \langle p,n_i\rangle + a_i$ for facet normals
\begin{align*}
&n_1 = -e_2-e_5,\ n_2 = -e_4+e_5,\ n_3=e_2,\ n_4=e_4,\\
&n_5=-e_1-e_5,\ n_6 = e_1,\ n_7=-e_3+e_5,\ n_8 =e_3.
\end{align*}
\item The $8 \times 8$ matrix $(h_i(m_j))$ equals \eqref{graphmatrix} and consists of the first eight rows of \eqref{graphHorn}.
\item The normal fan of $P$ is not simplicial.
\item $P$ has four primitive collections
\[
\{n_1,n_2, n_3,n_4\},\ \{n_5,n_6, n_7,n_8\},\ \{n_1,n_3, n_7,n_8\},\ \{n_2,n_4, n_5,n_6\}.
\]
\end{enumerate}
\end{proposition}

\begin{remark} 
The labels $m_1,\dots,m_8$ and $F_1,\dots,F_8$ of the lattice points and facets was chosen to ensure that $(h_i(m_j))$ was equal to \eqref{graphmatrix}; other labelings would have produced a permutation of the rows and columns of  \eqref{graphmatrix}.
\end{remark}

\begin{proof}
The join $Q_1\star Q_2$ is the union of all line segments connecting a point of $Q_1$ to a point of $Q_2$, where this is done so that distinct line segments have disjoint interiors. If we write $\R^5 = \mathrm{Span}(e_1,e_2) \times \mathrm{Span}(e_3,e_4) \times \R = V \times W \times \R$, then $P$ is the convex hull of $Q_1 \subseteq V \times \{0\} \times \{0\}$ and $Q_2 \subseteq \{0\} \times W \times \{1\}$.  This makes it easy to see that $P = Q_1\star Q_2$.

The polytope $P$ lies between the hyperplanes $t_5 = 0$ and $t_5 = 1$, so that its lattice points lie in the faces $Q_1$ and $Q_2$ where these hyperplanes meet $P$.  Since $Q_1$ and $Q_2$ are unit squares, their lattice points are their vertices.  Hence we get the eight lattice points listed in the statement of the proposition.  Note that the lattice points of $P$ are precisely its vertices.

A theorem of Kalai \cite{kalai} implies that the face lattice of $P= Q_1\star Q_2$ is the cartesian product of the face lattices of $Q_1$ and $Q_2$.  In particular, every facet of $P$ is of the form 
\begin{equation}
\label{facetsP}
E_1 \star Q_2,\ E_1 \text{ an edge of } Q_1,\quad
Q_1 \star E_2,\ E_2 \text{ an edge of } Q_2, 
\end{equation}
since the facets of $Q_1$ and $Q_2$ are edges.  It follows that there are eight facets.  We omit the straightforward proof of the facet presentation of $P$, though for later purposes, we note that  $h_i(p) = \langle p,n_i\rangle + a_i$ is given by 
\begin{equation}
\label{graphicalni}
\begin{array}{ll}
h_1(p) = 1-t_2-t_5,\, n_1 = -e_2-e_5 & \!h_5(p) = 1-t_1-t_5,\, n_5 = -e_1-e_5\\
h_2(p) = -t_4+t_5,\, n_2 = -e_4+e_5 & \!h_6(p) = t_1,\, n_6 = e_1\\
h_3(p) = t_2,\, n_3 = e_2 & \!h_7(p) = -t_3+t_5,\, n_7 = -e_3+e_5\\
h_4(p) = t_4,\, n_4 = e_4 & \!h_8(p) = t_3,\, n_8 = e_3.
\end{array}
\end{equation}
With this labeling, one computes that $(h_i(m_j))$ is the matrix \eqref{graphmatrix}.

To find the primitive collections of $P$, we need to understand the maximal cones of the normal fan.  Let $\sigma_j$ be the cone of the vertex $m_j$, which is generated by the facet normals of the facets $F_i$ containing $m_j$.  Looking at the matrix $(h_i(m_j))$, we see that
\begin{align*}
&m_1 \text{  is contained in all but } F_1,F_5, \text { so } \sigma_1 = \mathrm{Cone}(n_i \mid i \ne 1,5) \\
&m_2 \text{  is contained in all but } F_1,F_6, \text { so } \sigma_2 = \mathrm{Cone}(n_i \mid i \ne 1,6) \\
&m_3 \text{  is contained in all but } F_2,F_7, \text { so } \sigma_3 = \mathrm{Cone}(n_i \mid i \ne 2,7) \\
&m_4 \text{  is contained in all but } F_2,F_8, \text { so } \sigma_4 = \mathrm{Cone}(n_i \mid i \ne 2,8) \\
&m_5 \text{  is contained in all but } F_3,F_5, \text { so } \sigma_5 = \mathrm{Cone}(n_i \mid i \ne 3,5) \\
&m_6 \text{  is contained in all but } F_3,F_6, \text { so } \sigma_6 = \mathrm{Cone}(n_i \mid i \ne 3,6) \\
&m_7 \text{  is contained in all but } F_4,F_7, \text { so } \sigma_7 = \mathrm{Cone}(n_i \mid i \ne 4,7) \\
&m_8 \text{  is contained in all but } F_4,F_8, \text { so } \sigma_8 = \mathrm{Cone}(n_i \mid i \ne 4,8).
\end{align*} 
Using this, one can verify that the four quadruples listed in the  proposition are primitive collections. The cones $\sigma_j$ are $5$-dimensional, each with six minimal generators.  This proves that the normal fan of $P$ is not simplicial.

Now let $\mathcal{P}$ be a primitive collection, which we think of as a subset of $\{n_1,\dots,n_8\}$.   Then $\mathcal{P} \not\subseteq \sigma_j$ for all $j$, so that
\[
\begin{aligned}
&n_1 \text{ or } n_5 \in \mathcal{P},\ n_1 \text{ or } n_6 \in \mathcal{P},\ n_2 \text{ or } n_7 \in \mathcal{P},\ n_2 \text{ or } n_8 \in \mathcal{P},\\
&n_3 \text{ or } n_5 \in \mathcal{P},\ n_3 \text{ or } n_6 \in \mathcal{P},\ n_4 \text{ or } n_7 \in \mathcal{P},\ n_4 \text{ or } n_8 \in \mathcal{P}.
\end{aligned} 
\]
This easily implies that
\begin{equation}
\label{graphPC}
\{n_1,n_3\} \subseteq  \mathcal{P}  \text{ or } \{n_5,n_6\} \subseteq \mathcal{P} \ \ \text{and}\ \ \{n_2,n_4\} \subseteq  \mathcal{P}  \text{ or } \{n_7,n_8\} \subseteq \mathcal{P}. 
\end{equation}
{}From here, it is straightforward to show that the four primitive collections given above are in fact all of the primitive collections.
\end{proof}

We have three remaining tasks in our discussion of the graphical model:
\begin{itemize}
\item Explain why $P$ does not have strict linear precision. 
\item Show that $P$ has rational linear precision and find its Horn parametrization.
\item Explain the negative rows in the Horn matrix \eqref{graphHorn}.
\end{itemize} 

There are two ways to see that $P$ does not have strict linear precision with weights $w = (1,\dots,1)$.  First, the minimal Horn matrix \eqref{graphHorn} is not of the form \eqref{Hdef}, so $P$ cannot have strict linear precision by Theorem~\ref{XML}.  Second, Theorem~\ref{XML} tells us that strict linear precision is equivalent to $n_P=0$ and $\beta_w(p) = \text{constant}$.  The former is true---this follows easily from \eqref{graphicalni}---but the latter fails since for $p = (t_1,t_2,t_3,t_4,t_5)$, the formulas of \eqref{graphicalni} imply that
\begin{equation}
\label{graphbeta}
\begin{aligned}
\beta_w(p) =\ &h_1(p)h_5(p) + h_1(p)h_6(p) +h_2(p)h_7(p) +h_2(p)h_8(p)\, +\\ &h_3(p)h_5(p) +h_3(p)h_6(p) +h_4(p)h_7(p) +h_4(p)h_8(p)\\
=\ &2t_5^2-2t_5+1.
\end{aligned}
\end{equation}

To prove rational linear precision with weights $w = (1,\dots,1)$, we use the monomials $t^{m_j}$ coming from the lattice points $m_1,\dots,m_8$ listed in Propostion~\ref{Pproperties}.  The sum of these monomials is
\[
S(p) = \sum_{j=1}^8 t^{m_j} = 1+t_1+t_5+t_3t_5+t_2+t_1t_2+t_4t_5+t_3t_4t_5,
\]
so that for $p = (t_1,t_2,t_3,t_4,t_5)$, the monomial parametrization $\overline{w\chi}_\A$ of $Y_{\A,w}$ is given by
\begin{equation}
\label{graphparam}
\overline{w\chi}_\A(p) = \frac1{S(p)}\big(1,t_1,t_5,t_3t_5,t_2,t_1t_2,t_4t_5,t_3t_4t_5\big).
\end{equation}
Composing this with $\tau_\A: Y_{\A,w} \to \C^5$ gives 
\[
(\tau_\A\circ \overline{w\chi}_\A)(p) = \sum_{j=1}^s \frac{p^{m_j}}{S(p)}\hskip1pt m_j \in \C^5,
\]
which is birational with inverse
\[
\varphi(p) = \Big(\frac{h_6(p)}{h_5(p)}, \frac{h_3(p)}{h_1(p)},\frac{h_8(p)}{h_7(p)},\frac{h_4(p)}{h_2(p)},\frac{(1-t_5)h_2(p)h_7(p)}{t_5h_1(p)h_5(p)}\Big).
\]
By Proposition~\ref{rlpformula}, $P$ rational linear precision for weights $(1,\dots,1)$.  Furthermore, composing the inverse $\varphi$ with \eqref{graphparam} gives the parametrization 
\begin{equation}
\label{graphlpparam}
(\overline{w\chi}_\A \circ\varphi)(p) = \Big( \frac{h_1 h_5}{1-t_5}, \frac{h_1 h_6}{1-t_5}, \frac{h_2 h_7}{t_5}, \frac{h_2 h_8}{t_5},\frac{h_3 h_5}{1-t_5},\frac{h_3 h_6}{1-t_5},\frac{h_4 h_7}{t_5},\frac{h_4 h_8}{t_5}\Big)
\end{equation}
with rational linear precision.  Here, for simplicity we write $h_i$ instead of $h_i(p)$.  In constrast, the toric blending parametrization of $Y_{\A,w}$ is 
\begin{equation}
\label{graphkras}
\overline{w\beta}_\A(p) = \Big( \frac{h_1 h_5}{\beta_w}, \frac{h_1 h_6}{
\beta_w}, \frac{h_2 h_7}{\beta_w}, \frac{h_2 h_8}{\beta_w},\frac{h_3 h_5}{\beta_w},\frac{h_3 h_6}{\beta_w},\frac{h_4 h_7}{\beta_w},\frac{h_4 h_8}{\beta_w}\Big),
\end{equation}
where $\beta_w = 2t_5^2-2t_5+1$ by \eqref{graphbeta}.
Comparing \eqref{graphlpparam} and \eqref{graphkras}, we see that the denominators are the same, while the numerators differ.
For \eqref{graphlpparam}, notice that the denominators $t_5$ and $1-t_5$ can be explained as follows:
\begin{itemize}
\item $t_5$ occurs in the denominator in positions $3,4,7,8$, corresponding to the fact that $m_3,m_4,m_7,m_8$ have lattice distance $1$ to the hyperplane $t_5 = 0$.
\item $1-t_5$ occurs in the denominator in positions $1,2,5,6$, corresponding to the fact that $m_1,m_2,m_5,m_6$ have lattice distance $1$ to the hyperplane $1-t_5 = 0$.
\end{itemize}

Furthermore, if we substitute $p = \sum_{j=1}^8 u_j\hskip1pt m_j$ into \eqref{graphlpparam}, Proposition~\ref{rlpformula} gives the Horn parametrization from earlier in our discussion of the graphical model.

Finally, we need to see if the negative rows in the Horn matrix \eqref{graphHorn} have anything to do with the four primitive collections identified in Proposition~\ref{Pproperties}.  One of the primitive collections is $\{n_1,n_2,n_3,n_4\}$, and one can check that
\[
-(\text{row }1+\text{row }2+\text{row }3+\text{row }4) = (-1,-1,-1,-1,-1,-1,-1,-1).
\]
which is the row marked $?_1$ in  \eqref{graphHorn}.  However, all four primitive collections from Proposition~\ref{Pproperties} have this property.  So the primitive collections explain row $?_1$ but tell us nothing about rows $?_2$ and $?_3$.

This is not surprising, since the toric variety corresponding to $P$ is not simplicial by Proposition~\ref{Pproperties}.  Most results about primitive collections apply to simplicial toric varieties \cite[\S6.4]{cls}, though some results hold more generally \cite{CvR}.  

However, in the proof of Proposition~\ref{Pproperties}, we learned in \eqref{graphPC} that any primitive collection $\mathcal{P} \subseteq \{n_1,\dots,n_8\}$ satisfies
\[
\{n_1,n_3\} \subseteq  \mathcal{P}  \text{ or } \{n_5,n_6\} \subseteq \mathcal{P} \ \ \text{and}\ \ \{n_2,n_4\} \subseteq  \mathcal{P}  \text{ or } \{n_7,n_8\} \subseteq \mathcal{P}.
\]
Then observe that for the matrix \eqref{graphHorn}, we have
\begin{align*}
&-(\text{row }1+\text{row }3) = -(\text{row }5+\text{row }6) = (-1,-1,0,0,-1,-1,0,0)\\
&-(\text{row }2+\text{row }4) = -(\text{row }7+\text{row }8) = (0,0,-1,-1,0,0,-1,-1).
\end{align*}
These are precisely rows $?_2$ and $?_3$ from \eqref{graphHorn}.  It follows that the primitive collections of $P$ have an ``internal structure'' \eqref{graphPC} that explains how to get the final two rows of \eqref{graphHorn}. 

\subsection{Final Thoughts} For strict linear precision, we have the characterization given in Theorem~\ref{XML} and the suspicion voiced in Conjecture~\ref{simploidalconj} that the only polytopes with strict linear precision are the B\'ezier simpliods from Proposition~\ref{exampleprop}.

The situation for rational linear precision is less well understood.  While classifiying all polytopes with rational linear precision is likely to be a very difficult project, we suspect that it should be possible to describe the Horn matrix and Horn parametrization when $P$ has  rational linear precision.  

Here are some thoughts about the numerators and denominators that occur in the Horn parametrization.

\bigskip

\noindent {\scshape Numerators.} For the trapezoid, the  parametrization \eqref{trap2} with rational linear precision has the same numerators as the toric blending parametrization \eqref{trap3}, which is why the first four rows of the Horn matrix \eqref{trapHorn} are $(h_i(m_j))_{ij}$.  Similarly, for the graphical model, the parametrization with rational linear precision \eqref{graphlpparam} has the same numerators as the toric blending parametrization \eqref{graphkras}, which is why the first eight rows of the Horn matrix \eqref{graphHorn} are $(h_i(m_j))_{ij}$.  We conjecture that this happens in general, though proving this may require some new ideas. 

\bigskip

\noindent {\scshape Denominators.} Here, the situation is less clear.  
It is likely that primitive collections have a role to play when the normal fan is simplicial, though more examples will need to be computed before a precise conjecture can be formulated.  
In the nonsimplicial case, such as the graphical model, one may need to look deeper.  In general, we are convinced that the denominators in the Horn parametrization can be explained in terms of the polytope $P$, but the details remain elusive.

\appendix

\section{The Non-Negative Part of a Toric Variety}
\label{nonnegApp}

This appendix discusses the non-negative part $X_{\ge0}$ of a toric variety $X$.  We do not assume normality.  To begin, suppose we have an affine toric variety $V = \mathrm{Spec}(\C[\mathsf{S}])$, where $\mathsf{S} \subseteq M$ is an affine semigroup.  Here are two ways to define $V_{\ge0}$:
\begin{enumerate}
\item[1.] By \cite[Prop.\ 1.3.1]{cls}, points of $V$ correspond to semigroup homomorphisms $\mathsf{S} \to \C$, where $\C$ is a semigroup under multiplication.  Then $V_{\ge0} \subseteq V$ consists of points whose semigroup homomorphism takes values in $\R_{\ge0}$.
\item[2.] If $\mathsf{S} = \N\A \subseteq M$ with $\A$ finite, then $V = Y_\A$ is the affine toric variety of $\A \subseteq M$, i.e., $Y_\A$ is the Zariski closure of the image of the map
\begin{equation}
\label{affinetoricvariety}
t \in T \longmapsto (\chi^{m_1}(t),\dots, \chi^{m_s}(t)) \in \C^s
\end{equation}
where $T = \mathrm{Hom}_\Z(M,\C) \simeq (\C^\times)^d$ and $\A = \{m_1,\dots,m_s\}$.  Then $V_{\ge0}$ is the topological closure of the image of  $T_{>0} = \mathrm{Hom}_\Z(M,\R_{>0}) \simeq \R_{>0}^d$ in $\R_{\ge0}^s$.
\end{enumerate}
The two definitions of $V_{\ge0}$ are easily seen to be equivalent.  
Definition 2 is more common in the literature (see, for example, \cite{sottile}), while Definition 1 has the virtue of being functorial with respect to semigroup homomorphisms.  This has some nice consequences:
\begin{itemize}
\item If $X_\Sigma$ is the toric variety of a fan $\Sigma$, then for $\sigma \in \Sigma$, the subsets $(U_\sigma)_{\ge0} \subseteq U_\sigma$ are compatibles with inclusions of cones, so that 
\[
(X_\Sigma)_{\ge0} = \bigcup_{\sigma \in \Sigma} (U_\sigma)_{\ge0}.
\]
\item $X_\Sigma \mapsto (X_\Sigma)_{\ge0}$ is functorial with respect to toric morphisms.
\item For projective space, one can prove that
\begin{align*}
\PP^n_{\ge0} &= \{\text{points of $\PP^n$ represented by homogeneous coordinates  $\ge0$}\}\\
&= (\R_{\ge0}^{n+1}\setminus \{0\})/\R_{>0}.
\end{align*}
\item For $\A \subseteq M$ finite, the projective toric variety $X_\A \subseteq \PP^{|\A|-1}$ (possibly nonnormal) is covered by affine toric varieties with compatible semigroups, so $(X_\A)_{\ge0}$ can be defined in a functorial way similar to $(X_\Sigma)_{\ge0}$.
\item  As noted in \cite{sottile}, $(X_\A)_{\ge0}$ can be interpreted as a suitable topological closure in projective space using the projective version of \eqref{affinetoricvariety}.  
\end{itemize}

The norm map $|{\bullet}| : \C \to \R_{\ge0}$ induces a norm map $|{\bullet}| : X \to X_{\ge0}$ for any of the toric varieties considered above.  This is easy to see when $V = \mathrm{Spec}(\C[\mathsf{S}])$:\ if $p \in V$ is represented by $\mathsf{S}  \stackrel{p}{\to} \C$, then $|p| \in V_{\ge0}$ is represented by the composition 
\[
\mathsf{S} \longrightarrow \C \stackrel{|\bullet|}{\longrightarrow} \R_{\ge0}
\]
since the second map is a semigroup homomorphism.  By functorality, these patch to give $|{\bullet}| : X \to X_{\ge0}$.  This map is the identity on $X_{\ge0} \subseteq X$ and hence is surjective.

In the body of the paper, we use the norm-squared map $|{\bullet}|^2 : X \to X_{\ge0}$, which in the affine case sends $p \in V$ to $|p|^2 \in V_{\ge0}$ represented by the composition 
\begin{equation}
\label{normsquared}
\mathsf{S}  \stackrel{p}{\longrightarrow} \C \stackrel{|{\bullet}|}{\longrightarrow} \R_{\ge0} \stackrel{a\hskip1pt\mapsto a^2}{\longrightarrow} \R_{\ge0}.
\end{equation}
Observe that $|{\bullet}|^2 : X \to X_{\ge0}$ is onto since $a\mapsto a^2$ is a semigroup automorphism of $\R_{\ge0}$.  For $\C^n$ or $\PP^n$, $|{\bullet}|^2$ is the obvious map that sends each coordinate to its norm-squared.  

We close with two final observations:
\begin{itemize}
\item The norm map $|{\bullet}|$ and norm-squared map $|{\bullet}|^2$ are both functorial with respect to toric morphisms.
\item One can define the positive points $X_{>0} \subseteq X$, which lie in the torus $T \subseteq X$.  We have functorial maps $|{\bullet}| : T \to X_{>0}$ and $|{\bullet}|^2 : T \to X_{>0}$ 
\end{itemize}

\section*{Acknowledgements}

We are grateful to Eyal Markman for useful comments about the Fubini-Study metric.  We also thank Frank Sottile for suggesting the paper \cite{huh} by June Huh and Serkan Ho\c{s}ten for suggesting the paper \cite{seth} by Seth Sullivant.   Finally, we appreciate the helpful remarks and corrections provided by the referee.


\begin{thebibliography}{99}

\bibitem{hosten} C.\ Am\'endola, N.\ Bliss, I.\ Burke, C.\ Gibbons, M.\ Helmer, S.\ Ho\c{s}ten, E.\ Nash, J.\ Rodriguez and D.\ Smolkin, \emph{The maximum likelihood degree of toric varieties}, J. Symbolic Comput. {\bf 92} (2017), 222--242.


\bibitem{audin} M.\ Audin, \emph{Torus Actions on Symplectic Manifolds}, second revised edition, Birkh\"auser, Boston, MA, 2004.

\bibitem{cannas} A.\ Cannas da Silva, \emph{Symplectic Toric Manifolds}, Summer School on Symplectic Geometry of Integrable Hamiltonian Systems, Centre de Recerca Matem\`atica, Barcelona, 2001.  Available at 
\url{www.math.ethz.ch/~acannas/Papers/toric.pdf}.

\bibitem{cls} D.\ Cox, J.\ Little and H.\ Schenck, \emph{Toric Varieties}, AMS, Providence, RI, 2011. 

\bibitem{cox-homogeneous} D.\ Cox, \emph{The homogeneous coordinate ring of a toric variety}, J. Algebraic Geom. {\bf 4} (1995),  17--50


\bibitem{CvR} D.\ Cox and C.\ von Renesse, \emph{Primitive collections and toric varieties},  Tohoku Math.\ J.\ {\bf 61} (2009), 309--332.

\bibitem{DR} J.\ N.\ Darroch and D.\ Ratcliff, \emph{Generalized iterative scaling for log-linear models}, Ann.\ Math.\ Statist.\
{\bf 43} (1972), 1470--1480.

\bibitem{SG} L.\ Garcia-Puente and F.\ Sottile, \emph{Linear precision for parametric patches},  Adv.\ Comput.\ Math.\ {\bf 33} (2010), 191--214.

\bibitem{GMS} D.\ Geiger, C.\ Meek, and B.\ Sturmfels, \emph{On the toric algebra of graphical models}, Ann.\ Statist.\ {\bf 34} (2006), 1463--1492.

\bibitem{GRS} H.-C.\ Graf von Bothmer, K.\ Ranestad and F. Sottile, \emph{Linear precision for toric surface patches}, Found.\ Comput.\ Math.\ {\bf 10} (2010), 37--66.

\bibitem{guillemin} V.\ Guillemin, \emph{Kaehler structures on toric varieties},  J.\ Differential Geom.\ {\bf 40} (1994), 285--309. 
  
\bibitem{huh} J.\ Huh, \emph{Varieties with maximum likelihood degree one}, J.\ Algebr.\ Stat.\ {\bf  5} (2014), 1--17. 

\bibitem{HS} J.\ Huh and B.\ Sturmfels, \emph{Maximum Likelihood Geometry}, in \emph{Combinatorial Algebraic Geometry} (S.\ Di Rocco and B.\ Sturmfels, Eds.), Lecture Notes in Math.\ {\bf 2108}, Springer, New York, 2014, 63--117.

\bibitem{kalai} G.\ Kalai, \emph{A new basis of polytopes}, J.\ Combin.\ Theory Ser.\ A {\bf 49} (1988), 191--209.

\bibitem{krasauskas} R.\ Krasauskas, \emph{Toric surface patches}  Adv.\ Comput.\ Math.\ {\bf 17} (2002), 89--113.

\bibitem{lauritzen} S.\ Lauritzen, \emph{Graphical Models}, Oxford Univ.\ Press, Oxford, 1996.

\bibitem{marsden-weinstein} J.\ Marsden and A.\ Weinstein,  \emph{Reduction of symplectic manifolds with symmetry.} Rep. Mathematical Phys. {\bf 5} (1974), 121--130. 

\bibitem{meyer} K.\ Meyer, \emph{Symmetries and integrals in mechanics}, in \emph{Dynamical Systems {\rm}(Proc.\ Sympos., Univ.\ Bahia, Salvador, 1971{\rm)}} (M.\ Peixoto, ed.), Academic Press, New York, 1973, 259--272.

\bibitem{mobius} A.\ M\"obius. \emph{Der barycentrische Calcul},  J.\ A.\ Barth Verlag, Leipzig, 1827.
 
\bibitem{nick} N.\ Proudfoot, \emph{Lectures on Toric Varieties}, \url{pages.uoregon.edu/njp/tv.pdf}. 

\bibitem{sottile} F.\ Sottile, \emph{Toric ideals, real toric varieties, and the moment map}, in \emph{Topics in Algebraic Geometry and Geometric Modeling} (R. Goldman and R. Krasauskas, eds.), Contemp. Math. 334, AMS, Providence, RI, 2003, 225--240.

\bibitem{seth} S.\ Sullivant, \emph{Compressed polytopes and statistical disclosure limitation}, Tohoku Math.\ J.\ {\bf 58} (2006), 433--445.


 



\end{thebibliography}
\end{document}